\documentclass{amsart}%

\usepackage{booktabs}
\usepackage{hyperref}
\usepackage{xcolor}
\usepackage{amsaddr}
\usepackage{graphicx} 
\usepackage{float}
\usepackage[backend=biber,style=numeric]{biblatex}
\usepackage[margin=1.5in]{geometry}

\AtEveryBibitem{\clearfield{month}}
\AtEveryBibitem{\clearfield{url}}

\newtheorem{theorem}{Theorem}%
\newtheorem{proposition}[theorem]{Proposition}%

\newtheorem{example}{Example}%
\newtheorem{lemma}[theorem]{Lemma}%
\newtheorem{definition}{Definition}

\numberwithin{equation}{section}

\hypersetup{
    colorlinks,
    linkcolor={red!50!black},
    citecolor={blue!50!black},
    urlcolor={blue!80!black}
}
\hypersetup{
    colorlinks,
    linkcolor={red!},
    citecolor={blue!},
    urlcolor={blue!}
}

\usepackage{graphicx}%
\usepackage{multirow}%
\usepackage{amsmath,amssymb,amsfonts}%
\usepackage{amsthm}%
\usepackage{mathrsfs}%
\usepackage[title]{appendix}%
\usepackage{xcolor}%
\usepackage{textcomp}%
\usepackage{manyfoot}%
\usepackage{booktabs}%
\usepackage{algorithm}%
\usepackage{algorithmicx}%
\usepackage{algpseudocode}%
\usepackage{listings}%
\usepackage{diagbox}%
\usepackage{hyperref}%
\usepackage{float}

\begin{document}
\global\long\def\b#1{\left(#1\right)}%

\global\long\def\p{\phi}%

\global\long\def\abs#1{\left|#1\right|}%

\global\long\def\oiin{\oiint}%

\global\long\def\d{\text{d}}%

\global\long\def\curl{\nabla\times}%

\global\long\def\div{\nabla\cdot}%

\global\long\def\pd#1#2{\frac{\partial#1}{\partial#2}}%

\global\long\def\eval{\biggr|}%

\global\long\def\evals{\big|}%

\global\long\def\sb#1{\left[#1\right]}%

\global\long\def\cb#1{\left\{  #1\right\}  }%

\global\long\def\lc{\varepsilon}%

\global\long\def\R{\mathbb{R}}%

\global\long\def\C{\mathbb{C}}%

\global\long\def\qedd{\hfill\blacksquare}%

\global\long\def\code#1{\mathtt{#1}}%

\global\long\def\N{\mathbb{N}}%

\global\long\def\K{\mathbb{K}}%

\global\long\def\O{\mathcal{O}}%

\global\long\def\lb#1{\left(#1\right.}%

\global\long\def\rb#1{\left.#1\right)}%

\global\long\def\ss{\subseteq}%

\global\long\def\norm#1{\left\Vert #1\right\Vert }%

\global\long\def\ip#1{\left\langle #1\right\rangle }%

\global\long\def\B{\mathcal{B}}%

\global\long\def\ra{\mathcal{R}}%

\global\long\def\D{\mathcal{D}}%

\global\long\def\vp{\varphi}%

\global\long\def\sff#1#2{\left\langle #1,#2\right\rangle _{\text{II}}}%

\global\long\def\W{\mathcal{W}}%

\global\long\def\F{\mathcal{F}}%

\global\long\def\C{\mathbb{C}}%

\global\long\def\H{\mathcal{H}}%

\global\long\def\L{\mathcal{L}}%

\global\long\def\G{\mathcal{G}}%

\global\long\def\vc#1{\boldsymbol{#1}}%

\global\long\def\x{\vc x}%

\title[Meshfree Eigenvalues on Manifolds]{A Meshfree Method for Eigenvalues of Differential Operators
on Surfaces, Including Steklov Problems}

\author{Daniel R. Venn* and Steven J. Ruuth}
\address{Mathematics, Simon Fraser University, 8888 University Drive, Burnaby, V5A 1S6, BC, Canada}
\email{dvenn@sfu.ca}

\begin{abstract}We present and study techniques for investigating the spectra of linear
differential operators on surfaces and flat domains using symmetric
meshfree methods: meshfree methods that arise from finding norm-minimizing
Hermite--Birkhoff interpolants in a Hilbert space. Meshfree methods
are desirable for surface problems due to the increased difficulties associated with mesh creation and refinement on curved surfaces. 
While meshfree methods have been used for
solving a wide range of partial differential equations (PDEs) in recent years, the
spectra of operators discretized using radial basis functions (RBFs)
often suffer from the presence of non-physical eigenvalues (spurious
modes).
This makes many RBF methods unhelpful for eigenvalue problems.
We provide rigorously justified processes for finding eigenvalues
based on results concerning the norm of the solution in its native space;
specifically, only PDEs with solutions in the native space produce
numerical solutions with bounded norms as the fill distance approaches
zero. For certain problems, we prove that eigenvalue and eigenfunction
estimates converge at a high-order rate. The technique we present
is general enough to work for a wide variety of problems, including
Steklov problems, where the eigenvalue parameter is in the boundary
condition. Numerical experiments for a mix of standard and Steklov
eigenproblems on surfaces with and without boundary, as well as flat
domains, are presented, including a Steklov--Helmholtz problem.
\end{abstract}

\keywords{meshfree methods, eigenvalues, numerical analysis, radial basis functions, Steklov, spectral geometry}

\maketitle

\section{Introduction}

Meshfree methods are a class of numerical methods for differential
equations that differ from more traditional approaches, such as finite
elements, by not requiring the points used for computations to be organized.
Neighbours of points do not need to be specified, and the domain does
not need to be divided into simpler shapes, such as triangles. Instead,
all that is required is a sample of points in the domain of interest.
This is beneficial primarily because it is far easier to sample scattered points
in a domain than it is to form or refine a structured mesh of
the domain. Meshing can be particularly challenging on surface domains, especially those defined implicitly,
and it may be difficult to refine an existing mesh.
Point cloud generation, by contrast, is often straightforward, even for implicitly defined surfaces. Therefore, there
has been recent interest in developing and improving upon meshfree
methods for surface partial differential equations (PDEs), which appear in various applications, particularly
in image processing \cite{biddl13} and computer graphics \cite{auer12}.

For problems that simply require solving a well-posed PDE, a range
of successful, analyzed meshfree methods exist, with Hermite Radial
Basis Functions (RBFs) (see, for example, \cite{frank98,sun94}) among the best understood. Recent work has
also focused on understanding non-symmetric, least squares methods
using RBFs, including on surfaces \cite{chen20}. However, there are
a variety of problems in numerical analysis outside of solving PDEs.
Eigenvalue problems are particularly notable examples. For such problems,
literature on meshfree methods is more sparse, in part because it
is well known that common approaches to discretizing differential
operators using RBFs can produce incorrect, extra eigenvalues (spurious
modes). Discussions regarding spurious modes for RBFs
can typically be found in papers that focus on the stabilization of time-stepping
schemes \cite{fornb11,yan23}. Recently, Harlim et al. analyzed and
tested an RBF formulation that relies on Monte Carlo estimates of
surface integrals to produce a symmetric discretized Laplace--Beltrami
operator. The authors proved that this approach yields $\O\b{\tilde{N}^{-\frac{1}{2}}}$
convergence of the spectrum with high probability (see Theorem 4.1
of \cite{harli23}), provided the point cloud sampling density is known.

In this paper, we develop theory that can be used to produce reliable,
high-order, and flexible techniques for finding eigenvalues using
a range of meshfree methods. Specifically, we analyze methods that
can be developed from searching for a norm-minimizing Hermite--Birkhoff
interpolant in a Hilbert space. We show that the same method that
can be used for simple, 2D Laplacian eigenvalue computations can also
be used for Steklov and surface eigenvalue problems, all under the
same theoretical framework. In \cite{chand18}, the authors note that ``feasibility implies
convergence'' for these norm-minimizing methods. Proposition \ref{prop:Let--be} expands on this by showing that ``boundedness implies feasibility and convergence'', and it also shows a stronger form of convergence: convergence in the Hilbert space norm. For suitable choices of space, this implies uniform convergence of the solution and its derivatives up to a certain order. This result applies in the setting of PDEs on manifolds with meshfree methods minimizing
a norm in a Hilbert space, such as in symmetric RBF methods \cite{frank98,sun94}
and minimum Sobolev norm methods \cite{chand15,chand18}. This allows us to examine
boundedness to determine feasibility, which we use to develop
eigenvalue methods. We previously explored this approach in \cite{venn}, specifically for underdetermined Fourier extensions. In this paper, we significantly expand the theory behind the method, work in a more general setting so that the theoretical results apply to Hermite RBFs as well, and apply the method to a much wider range of problems.

For certain problems, we prove novel statements regarding convergence rates (Propositions \ref{prop:If-all-assumptions} and \ref{prop:If-all-assumptions-1} in Section \ref{sec:analysis}). More precisely, we show that our eigenvalue and eigenfunction estimates converge at arbitrarily high-order rates depending on the choice of Hilbert space, given that true eigenfunctions can be extended to sufficiently smooth functions in a larger domain.

In Section \ref{sec:Application-to-Eigenvalue}, we demonstrate the generality of our approach through
new numerical tests for a range of problems that may be difficult with existing
high-order methods. First, we expand on the test from \cite{venn} by estimating larger Laplace--Beltrami
eigenvalues. Then, we find Laplace--Beltrami eigenvalues on an implicitly-defined surface, Laplace eigenvalues on an irregular 2D domain with a hole, Steklov eigenvalues in both flat domains and on a surface
with boundary, Schrödinger--Steklov eigenvalues \cite{quino18}, and
Steklov--Helmholtz eigenvalues for a problem that produces singular
matrices with standard methods. Through our numerical tests, we demonstrate that the method is high-order, can be applied to implicit surfaces and other domains that are challenging to mesh, and requires minimal modification between problems. We conclude by summarizing our analysis contributions and numerical results in Section \ref{sec:conclude}.

\section{Background}\label{sec:bg}

\subsection{Functional Analysis Background\label{subsec:Functional-Analysis-Background}}

Certain meshfree methods, including Hermite RBFs \cite{frank98,sun94}
and minimum Sobolev norm interpolation \cite{chand15,chand18}, can
be analyzed as norm-minimizing interpolants in a Hilbert space. Our
analysis includes each of these methods, so we introduce
them in their most general formulation in this subsection.

Specifically, we consider norm-minimizing Hermite--Birkhoff interpolants:
functions that interpolate both function and derivative data. Let
$\H$ be a Hilbert space of functions on a domain $\Omega$. Let $\cb{\x_{j}}_{j=1}^{\tilde{N}}\subset\Omega$
be a collection of $\tilde{N}$ points. Let $\cb{\F_{j}}_{j=1}^{\tilde{N}}$ be
a collection of non-zero linear differential operators such that $\F_{j}$
is defined on a neighbourhood of $\x_{j}$. Define the evaluation
operator $\L:\H\to\C^{\tilde{N}}$ so that
\[
\b{\L u}_{j}:=\b{\F_{j}u}\b{\x_{j}}.
\]
We will end up requiring that $\L$ must be a bounded linear operator
(and that $\b{\F_{j}u}\b{\x_{j}}$ must have a uniquely defined value). This imposes a restriction on our choice of $\H$. More explicitly, we need constants $C_j>0$ such that $\abs{\F_j u\b \x_j}\le C_j\norm{u}_\H$ for all $u\in\H$. As a concrete example, consider $H^1\b{0,1}$. 
\\
\begin{example}
\label{exa:h1}If $\Omega=\b{0,1}$, $\H=H^{1}\b{0,1}$, and $\F_{j}$
are identity operators, then $\L$ is bounded. 

First note that $H^{1}$ functions on the interval $\b{0,1}$ can be uniquely
associated with a continuous function on $\sb{0,1}$ (see Exercise
5 in Chapter 5 of Evans \cite{evans10}, then note absolutely continuous
functions on $\b{0,1}$ have a unique continuous extension to $\sb{0,1}$)
so function evaluation can be well-defined. Let the minimum of $\abs u$
occur at $x_{*}\in\sb{0,1}$, then
\begin{align*}
\abs{u\b x} & =\abs{u\b{x_{*}}+\int_{x_{*}}^{x}u'\b t\,\d t}\\
 & \le\int_{0}^{1}\abs{u\b t}\,\d t+\int_{0}^{1}\abs{u'\b t}\,\d t\\
 & =\b{1,\abs u}_{L^{2}\b{0,1}}+\b{1,\abs{u'}}_{L^{2}\b{0,1}}\\
 & \le\norm u_{L^{2}\b{0,1}}+\norm{u'}_{L^{2}\b{0,1}}\text{, by the Cauchy-Schwarz inequality}\\
 & \le\sqrt{2}\norm u_{H^{1}\b{0,1}}.
\end{align*}
Then,
\[
\norm{\L u}_{2}\le\sqrt{2\tilde{N}}\norm u_{H^{1}\b{0,1}}.
\]
So $\L$ is bounded. $\hfill\qed$

\end{example}
~

A similar argument shows that if $u\in H^{p+1}\b{0,1}$, then $u$ can be uniquely associated with a $C^p$ function on $[0,1]$ and $\abs{\frac{\d^p}{\d x^p}u}\le\sqrt{2}\norm{u}_{H^{p+1}(0,1)}$. Then, as long as each $\F_j$ has order at most $p$, $\L:H^{p+1}(0,1)\to\C^{\tilde{N}}$ is bounded. Note that this is not the case in higher dimensions. For example, $H^1(\R^2)$ functions are not necessarily almost everywhere equal to a continuous function, so function evaluation cannot be well-defined everywhere, and $u\b\x$ cannot be bounded by a multiple of $\norm{u}_{H^1(\R^2)}$.

The setup for a variety of symmetric meshfree methods can be
summarized by the optimization problem:
\begin{align}
\text{minimize, over }u\in\H: & \,\norm u_{\H},\label{eq:opt1}\\
\text{subject to } & \,\L u=\vc f.\nonumber 
\end{align}
As long as $\H$ includes all compactly supported smooth functions,
all algebraic polynomials, all (possibly scaled) trigonometric polynomials,
or a variety of other classes of functions such that it is simple
to find one $u\in\H$ such that $\L u=\vc f$, then this problem has
a unique norm-minimizing solution $\tilde{u}$ to problem
(\ref{eq:opt1}) when $\L$ is bounded. This is since the constraint
set $\cb{u\in\H:\L u=\vc f}$ is closed when $\L$ is a bounded linear
operator. The norm-minimizing solution is orthogonal to the
kernel of $\L$; $u\in\mathcal{N}\b{\L}^{\perp}$. A key observation
is that $\L$ has a (bounded) adjoint operator $\L^{*}$ and that
$\mathcal{N}\b{\mathcal{L}}^{\perp}=\overline{\mathcal{R}\b{\mathcal{L}^{*}}}=\mathcal{R}\b{\mathcal{L}^{*}}$,
since $\mathcal{R}\b{\mathcal{L}^{*}}$ is finite dimensional ($\L^{*}:\C^{\tilde{N}}\to\H$,
so $\dim\b{\mathcal{R}\b{\mathcal{L}^{*}}}$ is at most $\tilde{N}$).
Therefore,
\begin{equation}
\tilde{u}\in\mathcal{R}\b{\mathcal{L}^{*}}=\text{span}\cb{\L^{*}\hat{\vc e}_{j}}_{j=1}^{\tilde{N}},\label{eq:finspan}
\end{equation}
where $\hat{\vc e}_{j}$ are the standard basis vectors for $\C^{\tilde{N}}$.
For any $\vc f\in\C^{\tilde{N}}$, there is therefore some $\vc{\beta}\in\C^{\tilde{N}}$
such that
\begin{align}
\L\L^{*}\vc{\beta}=\vc f,\label{eq:kernel}
\end{align}
and furthermore, this $\vc{\beta}$ is unique since the system is
square. It is then clear that $\L\L^{*}$ is self-adjoint and positive
definite. Finally, $\L^{*}\vc{\beta}$ is the solution to problem
(\ref{eq:opt1}).

\subsection{The Square System and RBFs}

An alternative view, common in RBF literature, is to start with a
set of functions $\cb{\psi_{j}}_{j=1}^{\tilde{N}}$  dependent on the
location of the points $\cb{\x_{j}}_{j=1}^{\tilde{N}}$. In the case
of RBF interpolation (without derivatives), it can be shown that many
standard choices of RBFs correspond to constrained norm-minimization
in a certain Hilbert space, typically called the native space (see
Theorem 13.2 of \cite{wendl04}). That is, the RBFs $\cb{\psi_{j}}_{j=1}^{\tilde{N}}$,
which are simply identical (up to translation) radially symmetric
functions centred at the points $\cb{\x_{j}}_{j=1}^{\tilde{N}}$,
are the functions $\cb{\L^{*}\hat{\vc e}_{j}}_{j=1}^{\tilde{N}}$
from the previous subsection. So, RBF interpolation can instead be
viewed as a highly underdetermined problem; the constraint set is
typically infinite-dimensional. RBFs simply select the optimal interpolant
by minimizing $\norm u_{\H}$ subject to the constraint. When derivative
interpolation conditions are included, the functions $\cb{\L^{*}\hat{\vc e}_{j}}_{j=1}^{\tilde{N}}$
instead correspond to Hermite RBFs $\cb{\psi_{j}}_{j=1}^{\tilde{N}}$, which are no longer radial but are related to the derivatives of the
usual RBFs \cite{frank98,sun94} (specifically, they
are $\F_{j}^{*}\phi\b{\x-\x_{j}}$ if $\phi$ is an RBF).

Constructing $\cb{\psi_{j}}_{j=1}^{\tilde{N}}$ can be advantageous,
as it turns an infinite-dimensional optimization problem into an $\tilde{N}\times\tilde{N}$
linear system ($\L\L^{*}\vc{\beta}=\vc f$). Numerically, however,
there are drawbacks. The condition number of $\L\L^{*}$ is large;
it is the square of the condition number of the original optimization
problem. There is therefore some motivation to solve or approximately
solve the optimization problem directly.

\subsection{Direct Solutions\label{subsec:Direct-Solutions}}

Of course, problem (\ref{eq:opt1}) cannot truly be solved directly
(that is, without forming the linear system $\L\L^{*}\vc{\beta}=\vc f$),
as it is an infinite-dimensional problem. However, if $\H$ is separable, then
we have an opportunity to reformulate the problem as an $\tilde{N}\times N_{b}$
underdetermined problem, for some number of basis functions $N_{b}>\tilde{N}$.
That is, suppose for some orthonormal basis $\cb{\phi_{n}}_{n=1}^{\infty}$
for $\H$, we can write
\[
\H=\cb{\sum_{n=1}^{\infty}c_{n}\phi_{n}:c_{n}\in\C\text{, for each \ensuremath{n\in\N}}}.
\]
Solving problem (\ref{eq:opt1}) more directly using a separable Hilbert
space is the approach taken in minimum Sobolev norm interpolation
\cite{chand18,chand15}. To do this, it is helpful to construct $\L$ in terms of $\cb{\phi_{n}}_{n=1}^{\infty}$
first. Recall that $\L$ is a bounded operator by assumption; $\norm{\L u}_{2}\le\norm{\L}\norm u_{\H}$ for any $u\in\H$.
In this case, each $\F_{j}\eval_{\x_{j}}$ must also be a bounded
operator from $\H\to\C$. Also note that if $u=\sum_{n=1}^{\infty}c_{n}\phi_{n}$,
then $\norm u_{\H}=\norm c_{\ell^{2}}$, since the $\phi_{n}$ functions
are assumed to be orthonormal. So, for each $c\in\ell^{2}$,
\[
\F_{j}\b{\sum_{n=1}^{\infty}c_{n}\phi_{n}\b{\x_{j}}}\le\norm{\L}\norm c_{\ell^{2}}.
\]
A short argument from the Riesz representation theorem shows $\cb{\F_{j}\phi_{n}\b{\x_{j}}}_{n=1}^{\infty}\in\ell^{2}$
with $\norm{\cb{\F_{j}\phi_{n}\b{\x_{j}}}_{n=1}^{\infty}}_{\ell^{2}}\le\norm{\L}$.
Furthermore,
\[
\sum_{n=N_{b}+1}^{\infty}\abs{c_{n}\b{\F_{j}\phi_{n}}\b{\x_{j}}}\le\sqrt{\b{\sum_{n=N_{b}+1}^{\infty}\abs{c_{n}}^{2}}\b{\sum_{n=N_{b}+1}^{\infty}\abs{\b{\F_{j}\phi_{n}}\b{\x_{j}}}^{2}}}\to0,
\]
as $N_{b}\to\infty$.
So, the sequence converges absolutely. Also, note that
\begin{align*}
\L & :\sum_{n=1}^{\infty}c_{n}\phi_{n}\mapsto\b{\sum_{n=1}^{\infty}c_{n}\b{\F_{j}\phi_{n}}\b{\x_{j}}}_{j=1}^{\tilde{N}},\\
\L^{*} & :\vc{\beta}\mapsto\sum_{n=1}^{\infty}\b{\sum_{j=1}^{\tilde{N}}\beta_{j}\b{\F_{j}^{*}\phi_{n}}\b{\x_{j}}}\phi_{n},\\
\L\L^{*} & :\vc{\beta}\mapsto\b{\sum_{n=1}^{\infty}\b{\sum_{j=1}^{\tilde{N}}\beta_{j}\b{\F_{j}^{*}\phi_{n}}\b{\x_{j}}}\b{\F_{k}\phi_{n}}\b{\x_{k}}}_{k=1}^{\tilde{N}}.
\end{align*}
We can then define a matrix $\vc{\Phi}$ associated with $\L\L^{*}$
by
\[
\Phi_{jk}:=\sum_{n=1}^{\infty}\b{\F_{j}^{*}\phi_{n}}\b{\x_{j}}\b{\F_{k}\phi_{n}}\b{\x_{k}}.
\]
This matrix is clearly self-adjoint when written in this
form. From the discussion at the end of Subsection \ref{subsec:Functional-Analysis-Background},
we also know that $\L\L^{*}$, and therefore $\vc{\Phi}$, is positive
definite as long as there is always at least one $u\in\H$ such that
$\L u=\vc f$ for any $\vc f\in\C^{\tilde{N}}$.

We may now consider the Hilbert space with a truncated basis:
\begin{align}\label{eq:Hsep}
\H_{N_{b}}:=\cb{\sum_{n=1}^{N_{b}}c_{n}\phi_{n}:c_{n}\in\C\text{ for each }n\in\cb{1,2,\ldots,N_{b}}}\subset\H.
\end{align}
The first question we may ask is whether we are still able to solve
$\L u=\vc f$ in $\H_{N_{b}}$. The next proposition tells us there
is always some finite $N_{b}$ such that a solution exists.
\\
\begin{proposition}
Assume $\H,\L$ are as defined in Subsection \ref{subsec:Functional-Analysis-Background}
and $\H_{N_{b}}$ is as defined as in Eq. \ref{eq:Hsep}. Let $\L_{N_{b}}:=\L\eval_{\H_{N_{b}}}$ be
bounded. If $\mathcal{R}\b{\L}=\C^{\tilde{N}}$, then there exists
some finite $M\in\N$ such that $\mathcal{R}\b{\L_{N_{b}}}=\C^{\tilde{N}}$ for all $N_{b}\ge M$.
\end{proposition}

\begin{proof}
We consider the matrix $\Phi$ associated with $\L\L^{*}$. It is
positive definite, so $\det\vc{\Phi}>0$. Let $\vc{\Phi}^{\b{N_{b}}}$
be the matrix associated with $\L_{N_{b}}\L_{N_{b}}^{*}$. Notice:
\begin{align*}
\abs{\Phi_{jk}-\Phi_{jk}^{\b{N_{b}}}} & =\abs{\sum_{n=N_{b}+1}^{\infty}\b{\F_{j}^{*}\phi_{n}}\b{\x_{j}}\b{\F_{k}\phi_{n}}\b{\x_{k}}}\\
 & \le\sqrt{\abs{\sum_{n=N_{b}+1}^{\infty}\abs{\b{\F_{j}\phi_{n}}\b{\x_{j}}}^{2}}\abs{\sum_{n=N_{b}+1}^{\infty}\abs{\b{\F_{k}\phi_{n}}\b{\x_{k}}}^{2}}}\\
 & \to0\text{, as }N_{b}\to0,
\end{align*}
where we recall that $\cb{\F_{j}\phi_{n}\b{\x_{j}}}_{n=1}^{\infty}\in\ell^{2}$.
Now, the determinant is a continuous function of the matrix entries,
so $\det\vc{\Phi}^{\b{N_{b}}}\to\det\vc{\Phi}$, and we must have
that there exists some $M>0$ such that $\det\vc{\Phi}^{\b{N_{b}}}>0$
for all $N_{b}\ge M$.
\end{proof}
An estimate for how many basis functions are needed for $\vc{\Phi}$
to be positive definite is given by Theorem 2.2 of \cite{chand18}.

Now that we know that $\L u=\vc f$ is solvable in $\H_{N_b}$ for sufficiently large $N_b$, we consider the truncated problem:

\begin{align*}
\text{minimize, over }u\in\H_{N_b}: & \,\norm u_{\H},\\
\text{subject to } & \,\L_{N_b} u=\vc f.
\end{align*}
Using $\cb{\phi_n}_{n=1}^{N_b}$ as an orthonormal basis for $\H_{N_b}$, this can be written in matrix form

\begin{align}\label{eq:linalgopt}
\text{minimize, over }\vc c\in\C^{N_b}: & \,\norm{\vc c}_2,\\
\text{subject to } & \,\vc V_{N_b} \vc c=\vc f,\nonumber
\end{align}
where
$\b{\vc V_{N_b}}_{jn} = \b{\F_j\phi_n}\b{\x_j}$.

Problem (\ref{eq:linalgopt}) is a common linear algebra problem with solution $\tilde{\vc c}:=\vc V_{N_b}^+\vc f$, where $\vc V_{N_b}^+$ is the Moore-Penrose pseudoinverse of $\vc V_{N_b}$. Numerically, $\tilde{\vc c}$ can be computed via singular value decomposition or complete orthogonal decomposition. We also have that $\vc \Phi_{N_b}=\vc V_{N_b}\vc V_{N_b}^*$ and $\vc V_{N_b}^+=\vc V_{N_b}^*\b{\vc V_{N_b}\vc V_{N_b}^*}^{-1}=\vc V_{N_b}^*\vc\Phi_{N_b}^{-1}$. However, it is important to note that the condition number of $\vc\Phi_{N_b}$ is the square of the condition number of $\vc V_{N_b}$; this means that using $\vc\Phi_{N_b}$ can result in a significantly higher error in practice for larger numbers of points ($\tilde{N}$). Therefore, when high accuracy is needed, computing $\tilde {\vc c}$ via singular value decomposition or complete orthogonal decomposition is preferred over solving $\vc \Phi_{N_b}\vc \beta=\vc f$ and using $\tilde{\vc c}=\vc V_{N_b}^*\vc \beta $. That said, if $N_b\gg \tilde{N}$ and high accuracy is not needed, or the number of points is low enough so that the condition number of $\vc \Phi_{N_b}$ is not too large, using $\vc \Phi_{N_b}$ can be significantly faster. This is particularly true if $\vc \Phi_{N_b}$ or terms in $\vc \Phi_{N_b}$ need to be computed only once for multiple problems of the form (\ref{eq:linalgopt}). Numerical details are discussed further in Subsection \ref{subsec:genimpl}.

\section{Analysis}\label{sec:analysis}
\subsection{Solvability and Convergence}

We continue with the setup introduced in Subsection \ref{subsec:Functional-Analysis-Background},
but introduce additional structure in order to state results for PDEs.
Let $S^{\b k}\ss\Omega$ be a domain for each $k\in\cb{1,\ldots,N_{S}}$
and let $\mathcal{G}^{\b k}$ be a differential operator defined on
a neighbourhood of $S^{\b k}$. $S^{\b k}$ may have any shape or
co-dimension. For example, $S^{\b 1}$ may be a surface, and we may
have $S^{\b 2}=\partial S^{\b 1}$. Let $\cb{\x_{j}}_{j=1}^{\tilde{N}}$be
a collection of points in $\bigcup_{k=1}^{N_{S}}S^{\b k}$, and let
$\F_{j}=\G^{\b{k_{j}}}$ for some $k_{j}\in\cb{1,\ldots,N_{S}}$ so
that $\x_{j}\in S^{\b{k_{j}}}$, where $\F_{j}$
is as introduced in Subsection \ref{subsec:Functional-Analysis-Background}
but now associated with an operator ($\G^{\b{k_{j}}}$) on a subdomain
($S^{\b{k_{j}}}$). 
The (non-discretized) problem we are interested
in is to find some $u\in\H$ such that
\begin{equation}
\G^{\b k}u=g^{\b k}\text{ on }S^{\b k}\text{, for each }k\in\cb{1,2,\ldots,N_{S}}.\label{eq:problemmain}
\end{equation}
We assume that the operators $\G^{\b k}$ are bounded in the $\H\to L^{\infty}$
sense on their domains. That is, we require, for all $\x\in S^{\b k}$
and for any $u\in\H$, that there exists some $M^{\b k}>0$ so that
\begin{equation}\label{eq:Gbound}
\abs{\b{\G^{\b k}u}\b{\x}}\le M^{\b k}\norm u_{\H}.
\end{equation}
Recalling Example \ref{exa:h1}, if $\G^{\b k}$
is the identity and $\H=H^{1}\b{0,1}$, for instance, then $M^{\b k}=\sqrt{2}$ satisfies the above inequality. For higher dimensions and higher derivatives,
a greater degree of smoothness is required; if $\H=H^{\frac{m}{2}+p+\lc}\b{\R^{m}}$
for some $p\in\N$ and $\lc>0$, then such a constant exists when $\G^{\b k}=\partial^{\alpha}$
for $\abs{\alpha}\le p$ (see Theorem 3.26 of \cite{mclea00}).

With the assumption given by Eq. (\ref{eq:Gbound}), if $\H$ is separable, we have
\begin{align*}
\abs{\sum_{n=1}^{\infty}c_{n}\b{\G^{\b k}\phi_{n}}\b{\x}} & \le M^{\b k}\norm c_{\ell^{2}},\\
\abs{\sum_{n=N_{b}+1}^{\infty}c_{n}\b{\G^{\b k}\phi_{n}}\b{\x}} & \le M^{\b k}\sqrt{\sum_{n=N_{b}+1}^{\infty}\abs{c_{n}}^{2}}.
\end{align*}
So, the series $\sum_{n=1}^{\infty}c_{n}\b{\G^{\b k}\phi_{n}}$ converges
uniformly on $S^{\b k}$ as $N_{b}\to\infty$. In particular, if $\cb{\G^{\b k}\phi_{n}}$
are continuous on $S^{\b k}$, then $\G^{\b k}u=\sum_{n=1}^{\infty}c_{n}\b{\G^{\b k}\phi_{n}}$
will be as well. Therefore, the boundedness of $\abs{\b{\G^{\b k}u}\b{\x}}$
by a constant multiple of $\norm u_{\H}$ is enough for $\G^{\b k}u$
to be continuous on $S^{\b k}$ for all $u\in\H$ in the case that $\H$
is separable. For each $k$, we also let $g^{\b k}$ be a function on $S^{\b k}$
so that $f_{j}=g^{\b{k_{j}}}\b{\x_{j}}$, where the index $k_j$ is defined such that $\F_{j}=\G^{\b{k_{j}}}$.
This leads to an important proposition, which is a generalized version of a result applying to Hilbert spaces constructed from Fourier series in \cite{venn}.
\\
\begin{proposition}
\label{prop:Let--be}Let $\cb{\x_{j}}_{j=1}^{\infty}$ be a collection
of points such that $\cb{\x_{j}}_{j=1}^{\infty}\cap S^{\b k}$ is
dense in $S^{\b k}$ for each $k\in\cb{1,2,\ldots,N_{S}}$. Let $\L^{\b{\tilde{N}}}:\H\to\C^{\tilde{N}}$
such that $\b{\L^{\b{\tilde{N}}}u}_{j}=\F_{j}u\b{\x_{j}}$ and assume
that each $\L^{\b{\tilde{N}}}$ is bounded and that $\G^{\b k}u$ is continuous
on $S^{\b k}$ for all $u\in\H$ and for each $k\in\cb{1,2,\ldots,N_{S}}$. Then, if it exists, let $u^{\b{\tilde{N}}}$
be the unique solution to
\begin{align}
\text{minimize, over }u\in\H: & \,\norm u_{\H},\label{eq:opt1-1}\\
\text{subject to } & \,\F_{j}u\b{\x_{j}}=f_{j}\text{, for each \ensuremath{j\in\cb{1,2,\ldots,\tilde{N}}}}.\nonumber 
\end{align}
If for each $\tilde{N}$, $u^{\b{\tilde{N}}}$ exists and $\norm{u^{\b{\tilde{N}}}}_{\H}<B$
for some $B>0$, then there exists a solution $u^{\b{\infty}}\in\H$
to
\begin{equation}
\G^{\b k}u^{\b{\infty}}=g^{\b k}\text{ on }S^{\b k}\text{, for each }k\in\cb{1,2,\ldots,N_{S}}.\label{eq:fullpde}
\end{equation}
Furthermore, $\norm{u^{\b{\infty}}-u^{\b{\tilde{N}}}}_{\H}\to0$.
\end{proposition}

\begin{proof} Assume that the solutions $u^{\b{\tilde{N}}}$ to problem (\ref{eq:opt1-1}) exist for each $\tilde{N}$ and that $\norm{u^{\b{\tilde{N}}}}_{\H}<B$ for some $B>0$.
First note that if $\tilde{N}_{1}<\tilde{N}_{2}$, then $\F_{j}u^{\b{\tilde{N}_{2}}}\b{\x_{j}}=f_{j}\b{\x_{j}}$
for each $j\in\cb{1,2,\ldots,\tilde{N}_{2}}\supset\cb{1,2,\ldots,\tilde{N}_{1}}$,
so $\norm{u^{\b{\tilde{N}_{1}}}}_{\H}\le\norm{u^{\b{\tilde{N}_{2}}}}_{\H}$.
Furthermore, $u^{\b{\tilde{N}_{1}}}-u^{\b{\tilde{N}_{2}}}\in\mathcal{N}\b{\L^{\b{\tilde{N}_{1}}}}$, and $u^{\b{\tilde{N}_{1}}}\in\mathcal{N}\b{\L^{\b{\tilde{N}_{1}}}}^{\perp}$,
so
\begin{align*}
\b{u^{\b{\tilde{N}_{1}}},u^{\b{\tilde{N}_{1}}}-u^{\b{\tilde{N}_{2}}}}_{\H} & =0\\
\implies\norm{u^{\b{\tilde{N}_{1}}}}_{\H}^{2} & =\b{u^{\b{\tilde{N}_{1}}},u^{\b{\tilde{N}_{2}}}}_{\H}=\b{u^{\b{\tilde{N}_{2}}},u^{\b{\tilde{N}_{1}}}}_{\H}.
\end{align*}

Then, $\cb{\norm{u^{\b{\tilde{N}}}}_{\H}}$ is a bounded,
non-decreasing sequence of real numbers, so $\norm{u^{\b{\tilde{N}}}}_{\H}\uparrow\tilde{B}$,
for some $\tilde{B}\ge0$. Then,
\begin{align*}
\norm{u^{\b{\tilde{N}_{1}}}-u^{\b{\tilde{N}_{2}}}}_{\H}^{2} & =-\b{u^{\b{\tilde{N}_{2}}},u^{\b{\tilde{N}_{1}}}-u^{\b{\tilde{N}_{2}}}}_{\H}\\
 & =\norm{u^{\b{\tilde{N}_{2}}}}_{\H}^{2}-\b{u^{\b{\tilde{N}_{2}}},u^{\b{\tilde{N}_{1}}}}_{\H}\\
 & =\norm{u^{\b{\tilde{N}_{2}}}}_{\H}^{2}-\norm{u^{\b{\tilde{N}_{1}}}}_{\H}^{2}\\
 & \le2\tilde{B}\b{\norm{u^{\b{\tilde{N}_{2}}}}_{\H}-\norm{u^{\b{\tilde{N}_{1}}}}_{\H}}.
\end{align*}
The sequence of norms $\cb{\norm{u^{\b{\tilde{N}}}}_{\H}}$ converges and is therefore
Cauchy, so $\cb{u^{\b{\tilde{N}}}}$ is Cauchy in $\H$ and therefore
converges to some $u^{\b{\infty}}\in\H$, since $\H$ is a Hilbert
space. Now, for each $\tilde{N}>\tilde{N}_{1}$, notice $u^{\b{\tilde{N}}}$
is in the constraint set $u^{\b{\tilde{N}_{1}}}+\mathcal{N}\b{\L^{\b{\tilde{N}_{1}}}}$,
which is closed since $\L^{\b{\tilde{N}_{1}}}$ is bounded. We then
have that $u^{\b{\infty}}\in u^{\b{\tilde{N}_{1}}}+\mathcal{N}\b{\L^{\b{\tilde{N}_{1}}}}$
for each $\tilde{N}_{1}$, and importantly,

\[
\F_{j}u^{\b{\infty}}\b{\x_{j}}=f\b{\x_{j}}\text{, for each \ensuremath{j\in\N}},
\]
and therefore,

\[
\G^{\b{k_{j}}}u^{\b{\infty}}\b{\x_{j}}=g^{\b{k_{j}}}\b{\x_{j}}\text{ on }S^{\b{k_{j}}}\text{, for each }j\in\N.
\]
Finally, $\cb{\x_{j}}_{j=1}^{\infty}\cap S^{\b k}$ is dense in $S^{\b k}$
and $\G^{\b k}u^{\b{\infty}}$ is continuous, so 

\[
\G^{\b k}u^{\b{\infty}}=g^{\b k}\text{ on }S^{\b k}\text{, for each }k\in\cb{1,2,\ldots,N_{S}}.
\]
\end{proof}
There are a few observations to make about this proposition. The first
is that the existence of any strong form solution $u\in\H$ to problem
(\ref{eq:problemmain}) would imply that $\norm{u^{\b{\tilde{N}}}}_{\H}$
is bounded. Perhaps more interesting from a numerical perspective
is that we can conclude the converse from the proposition. If $\norm{u^{\b{\tilde{N}}}}_{\H}$
is bounded, then problem (\ref{eq:problemmain}) is solvable. This is notable
because $\norm{u^{\b{\tilde{N}}}}_{\H}$ is a quantity that is simple
to compute numerically; it is $\sqrt{\vc{\beta}^{*}\vc{\Phi\beta}}=\sqrt{\vc{\beta}^{*}\vc f}$.
Just by examining $\norm{u^{\b{\tilde{N}}}}_{\H}$, we can then hope
to numerically investigate an extremely wide range of PDE solvability
questions, including eigenvalue problems. We also note that the interpolant solutions $u^{\b{\tilde{N}}}$ converge to $u^{\b{\infty}}$,
a solution to the PDE (\ref{eq:problemmain}), in the $\H$-norm. For suitable
choices of $\H$, this implies uniform convergence of $u^{\b{\tilde{N}}}$
and its derivatives of sufficiently low order (depending on $\H$) to $u^{\b{\infty}}$
on \textit{all of $\Omega$},
even in regions where there are no scattered points.

In the case that it is known that a solution to (\ref{eq:problemmain})
exists and can be extended to a function $u\in\H$, boundedness of
$\norm{u^{\b{\tilde{N}}}}_{\H}$ is guaranteed (since $u$ is feasible). This implies that 
$u^{\b{\tilde{N}}}$ will converge in $\H$ as $\tilde{N}\to\infty$ to some solution $u^{\b{\infty}}$ to
(\ref{eq:problemmain}). The solution $u^{\b{\infty}}$ may or may not be equal to $u$, since the
extension from a solution of the PDE to $\H$ may not be unique, and
the PDE may have multiple solutions on its domain. This is noteworthy
because it implies convergence of $u^{\b{\tilde{N}}}$ to a single,
norm-minimizing solution $u^{\b{\infty}}$, even in cases where there are multiple
solutions to the PDE. Also note that when a solution to (\ref{eq:problemmain}) exists, a result in Banach spaces, rather than just Hilbert spaces, shows that $\G^{\b k}u^{\b{\tilde{N}}}\to g^{\b k}$ pointwise (see Theorem 7.1 of \cite{chand18}); Proposition \ref{prop:Let--be} expands on this result in Hilbert spaces by noting that boundedness of $\norm{u^{\b{\tilde{N}}}}_\H$ implies the existence of a solution $u^{\b{\infty}}$ to (\ref{eq:problemmain}), and that $u^{\b{\tilde{N}}}\to u^{\b \infty}$ in the $\H$-norm.

\subsection{Convergence Rate Analysis for Certain Eigenvalue Problems\label{subsec:Convergence-Rate-Analysis}}

To state convergence rates, we must first recall the definition
of the fill distance.
~

\begin{definition}
The fill distance $h_{\text{max}}$ of a collection
of points $\cb{\x_{j}}_{j=1}^{\tilde{N}}\subset S$, where $S\subset\R^{m}$
is a bounded domain, is given by
\[
h_{\text{max}}=\sup_{\x\in S}\min_{j\in\cb{1,2,\ldots,\tilde{N}}}\norm{\x-\x_{j}}_2.
\]

\end{definition}
That is, the fill distance is the largest distance
between a point $\x\in S$ and its closest point in $\cb{\x_{j}}_{j=1}^{\tilde{N}}$.

We now set up an eigenvalue problem, for which we
will prove the convergence of eigenvalue and eigenfunction estimates.
Let $\cb{\vc a_{j}}_{j=1}^{\tilde{N}_{a}}$ be a collection of points
on a domain $S$ or its boundary $\partial S$, and let $\cb{b_{j}}_{n=1}^{\tilde{N}_{a}}$
be a set of values such that at least one $b_{j}$ is non-zero. We
then consider the eigenvalue problem:
\begin{align}
\mathcal{F}u-\lambda u & =0\text{ on \ensuremath{S}},\label{eq:fulleigprob}\\
\G u\eval_{\partial S} & =0,\nonumber \\
u\b{\vc a_{j}} & =b_{j}\text{, for each \ensuremath{j\in\cb{1,2,\ldots,\tilde{N}_{a}}}},\nonumber 
\end{align}
where $\F$ is a linear differential operator of order $q$ and $\mathcal{G}$ is a differential operator
of order $\tilde{q}\le q$.
Let $C_{\G}^{q}\b{\overline{S}}=\cb{u\in C^{q}\b{\overline{S}}:\G u\eval_{\partial S}=0}$. 

\subsubsection{Assumptions\label{subsec:Assumptions}}

We make the following list of assumptions regarding the problem. Note
that each of these is satisfied by an elliptic $\F$, such as the Laplacian, with sufficiently
smooth coefficient functions on a Lipschitz domain
and a $\G$ that corresponds to Dirichlet, Neumann, or Robin boundary conditions (see Thms. 4.10-4.12 of \cite{mclea00} for symmetric, strongly elliptic operators on Lipschitz domains, for example). These assumptions are also satisfied by the Laplace--Beltrami operator on smooth manifolds, either on closed manifolds or with Dirichlet, Neumann, or Robin boundary conditions (see Chapter 5 of the text by Taylor \cite{taylo23} for regularity theory for the Laplace--Beltrami operator). Importantly, each (non-Steklov) problem examined in the numerical experiments in Section 4 satisfies these assumptions. 
\begin{enumerate}
\item \label{enu:well-posed}There exists some $\alpha_{k}\in\C$
so that for any $C^{\tilde{q}}$ function $w$, the related problem
with inhomogeneous boundary conditions can be made well-posed:
\begin{align*}
\b{\F-\alpha_{k}}g & =0\text{ on $S$},\\
\G g\eval_{\partial S} & =\G w\eval_{\partial S}.
\end{align*}
Specifically, we require that there exists a constant $C_{k}>0$ such
that the problem above always has at least one solution $g\in C^{q}\b{\overline{S}}$
such that, for some $\tilde{s}\in\R$,
\begin{equation}
\norm g_{L^{2}\b S}\le C_{k}\norm{\G w}_{H^{\tilde{s}}\b{\partial S}}.\label{eq:gassump}
\end{equation}
This assumption holds for Lipschitz domains $S$, uniformly elliptic
operators $\F-\alpha_{k}$, and Dirichlet, Neumann, or Robin boundary
conditions (see Thms. 4.10 and 4.11 of \cite{mclea00}, noting $\tilde{s}=\frac{1}{2}$
for Dirichlet problems and $\tilde{s}=-\frac{1}{2}$ for Neumann and
Robin problems), assuming $\alpha_{k}$ is chosen such that it is
not an eigenvalue of $\F$.
\item \label{enu:counteig}There is a countable set of values of $\lambda$
for which a non-zero solution $u\in C^{q}\b S$ exists to (\ref{eq:fulleigprob}).
We call this set $\cb{\lambda_{j}}$: the set of eigenvalues.
\item \label{enu:closedeig}For a specific $\lambda_{k}$ of interest, the
eigenspace $\mathcal{N}\b{\F_{\G}-\lambda_{k}}$, where $\F_{\G}:C_{\G}^{q}\b{\overline{S}}\to C\b S$
is the restriction of $\F$ to $C_{\G}^{q}\b{\overline{S}}$, is closed
as a subspace of $L^{2}\b S$ (finite-dimensional, for example). 
\item \label{enu:lambda}$\lambda$ is closer to $\lambda_{k}$ than to
any other eigenvalue in $\cb{\lambda_{j}}$.
\item \label{enu:coercive}$\F_{\G}-\lambda$ is coercive on $\mathcal{N}\b{\F_{\G}-\lambda_{k}}^{\perp}\cap C_{\G}^{q}\b{\overline{S}}$
with the same constant for each $\lambda$ satisfying point \ref{enu:lambda}.
Specifically, there exists a $B_{k}>0$ such that for all $\lambda$
satisfying point \ref{enu:lambda} and all $u\in\mathcal{N}\b{\F_{\G}-\lambda}^{\perp}\cap C_{\G}^{q}\b S$,
$B_{k}\norm u_{L^{2}\b S}\le\norm{\b{\F_{\G}-\lambda}u}_{L^{2}\b S}$. One case where this holds is when $\F_\G$ admits an orthonormal basis of eigenfunctions (in this case, $B_k$ is just half the spectral gap for $\F_\G - \lambda_k$), $\F_\G$ has no accumulation points in its spectrum, and there exists an $\alpha\in\C$ such that $\F_\G - \alpha$ has a bounded inverse with respect to the $L^2$ norm.

\item \label{enu:angle}There is some $\theta_{k}\in\left\lceil 0,1\right)$ such that for all $v\in\mathcal{N}\b{\F_{\G}-\lambda_{k}}^{\perp}\cap C_{\G}^{q}\b{\overline{S}}$
and all $u\in\mathcal{N}\b{\F_{\G}-\lambda_{k}}$, 
\[
\abs{\b{\b{\F_{\G}-\lambda}v,u}_{L^{2}\b S}}\le\theta_{k}\norm{\b{\F_{\G}-\lambda}v}_{L^{2}\b S}\norm u_{L^{2}\b S}.
\]
It is straightforward to show
in this case that 
\[
\norm{\b{\F_{\G}-\lambda}v+u}_{L^{2}\b S}^{2}\ge\b{1-\theta_{k}}\b{\norm{\b{\F_{\G}-\lambda}v}_{L^{2}\b S}^{2}+\norm u_{L^{2}\b S}^{2}}.
\]
Note that if $\F_{\G}$ is a symmetric operator, $\theta_{k}=0$
and the inequality becomes an equality.

\end{enumerate}
~

Finally, we also consider the discretized problem. Let $\cb{\x_{j}}_{j=1}^{\infty}\subset S$
and $\cb{\vc y_{j}}_{j=1}^{\infty}\subset\partial S$ be point clouds
such that the fill distance $h_{\text{max}}$ of $\cb{\x_{j}}_{j=1}^{\tilde{N}}$
in $S$ and $\cb{\vc y_{j}}_{j=1}^{\tilde{N}_{\partial}}$ in $\partial S$ goes to
zero as $\tilde{N},\tilde{N}_{\partial}\to\infty$. The discretized
problem is typically:
\begin{align}
\text{minimize, over }u\in\H:\, & \norm u_{\H},\label{eq:opt1-1-1}\\
\text{subject to }\, & \mathcal{F}u\b{\x_{j}}-\lambda u\b{\x_{j}}=0\text{, for each \ensuremath{j\in\cb{1,2,\ldots,\tilde{N}}}},\nonumber \\
 & \mathcal{G}u\b{\vc y_{j}}=0\text{, for each \ensuremath{j\in\cb{1,2,\ldots,\tilde{N}_{\partial}}}},\nonumber \\
 & u\b{\vc a_{j}}=b_{j}\text{, for each \ensuremath{j\in\cb{1,2,\ldots,\tilde{N}_{a}}}}.\nonumber 
\end{align}
Note that splitting any of these conditions into
two or more, such as $\mathcal{F}_{1}u\b{\x_{j}}-\lambda u\b{\x_{j}}=0$
and $\F_{2}u\b{\x_{j}}=0$, where $\F=\F_{1}+\F_{2}$, does not change
the proof or conclusions of Proposition \ref{prop:If-all-assumptions}. This is noteworthy since this is how
surface PDEs are discretized later (see Subsection \ref{subsec:Laplace--Beltrami} and Eq. (\ref{eq:eiglapbeltdisc}) for the discretization of a surface PDE).

We assume $\H$ is such that there are constants $A,\tilde{A},p,Q_{a}>0$
such that the solution $\tilde{u}_{\lambda}$ to problem (\ref{eq:opt1-1-1}) exists and
satisfies, for small enough $h_{\text{max}}$, 
\begin{align}
\norm{\mathcal{F}\tilde{u}_{\lambda}-\lambda\tilde{u}_{\lambda}}_{L^{2}\b S} & \le Ah_{\text{max}}^{p}\norm{\tilde{u}_{\lambda}}_{\H},\label{eq:Fuassumpt}\\
\norm{\G\tilde{u}_{\lambda}}_{H^{\tilde{s}}\b{\partial S}} & \le\tilde{A}h_{\text{max}}^{p}\norm{\tilde{u}_{\lambda}}_{\H},\label{eq:Guassump}\\
\norm{\tilde{u}_{\lambda}}_{L^{2}\b S} & \ge\frac{\max\cb{\abs{b_{j}}}Q_{a}}{\norm{\tilde{u}_{\lambda}}_{\H}},\label{eq:normeq}
\end{align}
where
$\tilde{s}$ is the same as in assumption \ref{enu:well-posed}.
Equations (\ref{eq:Fuassumpt})--(\ref{eq:normeq}) are all readily obtainable for many operators $\F,\G$ and spaces
$\H$ used in meshfree methods. The main result (Theorem 1) of \cite{narco05}
leads to (\ref{eq:Fuassumpt}) for a wide range of $\H$-norms that
dominate a Sobolev norm, noting that $\mathcal{F}\tilde{u}_{\lambda}-\lambda\tilde{u}_{\lambda}$
is a function with scattered zeros. $(\ref{eq:Guassump})$ can also
be obtained from Theorem 1 of \cite{narco05} for boundaries that
are sufficiently piecewise smooth, if the result is applied to the
domain of the parametrization(s) of the boundary. The last condition
uses the fact that $\tilde{u}_{\lambda}\b{\vc a_{j}}=b_{j}$ and the
fact that it is often simple to show that $\norm{\nabla u}_{L^{\infty}\b{\Omega}}\le K\norm u_{\H}$
for some constant $K>0$ and various choices of $\H$.

\subsubsection{Convergence-Divergence Result for Well-Posed Eigenvalue Problem}
\begin{proposition}
\label{prop:If-all-assumptions}If all assumptions from Subsection
\ref{subsec:Assumptions} hold, then there exist constants $\tilde{B}_{k},\tilde{C}_{k},\Lambda_{k},\tilde{\Lambda}_{k}>0$
such that for small enough $h_{\text{max}}$, there exists an eigenfunction
$v_{k}\in\mathcal{N}\b{\F_{\G}-\lambda_{k}}$ such that, if $\b{\tilde{B}_{k}+\Lambda_{k}\abs{\lambda-\lambda_{k}}}h_{\text{max}}^{p}\norm{\tilde{u}_{\lambda}}_{\H}^{2}\le\max\cb{\abs{b_{j}}}Q_{a}$,
\begin{align*}
\abs{\lambda-\lambda_{k}} & \le\frac{\b{\tilde{C}_{k}+\tilde{\Lambda}_{k}\abs{\lambda-\lambda_{k}}}h_{\text{max}}^{p}\norm{\tilde{u}_{\lambda}}_{\H}^{2}}{\max\cb{\abs{b_{j}}}Q_{a}-\b{\tilde{B}_{k}+\Lambda_{k}\abs{\lambda_{k}-\lambda}}h_{\text{max}}^{p}\norm{\tilde{u}_{\lambda}}_{\H}^{2}},\\
\frac{\norm{\tilde{u}_{\lambda}-v_{k}}_{L^{2}\b S}}{\norm{v_{k}}_{L^{2}\b S}} & \le\frac{\b{\tilde{B}_{k}+\Lambda_{k}\abs{\lambda-\lambda_{k}}}h_{\text{max}}^{p}\norm{\tilde{u}_{\lambda}}_{\H}^{2}}{\max\cb{\abs{b_{j}}}Q_{a}-\b{\tilde{B}_{k}+\Lambda_{k}\abs{\lambda-\lambda_{k}}}h_{\text{max}}^{p}\norm{\tilde{u}_{\lambda}}_{\H}^{2}},
\end{align*}
where $\tilde{u}_{\lambda}$ is the solution to problem (\ref{eq:opt1-1-1}).
\end{proposition}

\begin{proof}
Let $g,\alpha_{k}$ be such that
\begin{align}
\b{\F-\alpha_{k}}g & =0,\label{eq:gassump-1}\\
\G g\eval_{\partial S} & =\G\tilde{u}_{\lambda}\eval_{\partial S},\nonumber \\
\norm g_{L^{2}\b S} & \le C_{k}\norm{\G\tilde{u}_{\lambda}}_{H^{\tilde{s}}\b{\partial S}},\label{eq:gnorm}
\end{align}
where we use Eq. (\ref{eq:gassump}) from assumption \ref{enu:well-posed}. Now,
let $v_{k}$ be the projection of $\tilde{u}_{\lambda}-g$ onto $\mathcal{N}\b{\F_{\G}-\lambda_{k}}$
in $L^{2}\b S$ so that
\[
\tilde{u}_{\lambda}  = g+v_{k}+v_{k}^{\perp},
\]
where $v_{k}^{\perp} =\tilde{u}_{\lambda}-g-v_{k}\in\mathcal{N}\b{\F_{\G}-\lambda_{k}}^{\perp}$. Then, 
$\tilde{u}_{\lambda}-g$ is in $C_{\G}^{q}\b S$ (we specifically
define $g$ so that this is the case), as is $v_{k}$, so $v_{k}^{\perp}$
is also in $C_{\G}^{q}\b S$. Now,
\begin{align}
\norm{\mathcal{F}\tilde{u}_{\lambda}-\lambda\tilde{u}_{\lambda}}_{L^{2}\b S} & =\norm{\mathcal{F}\b{g+v_{k}+v_{k}^{\perp}}-\lambda\b{g+v_{k}+v_{k}^{\perp}}}_{L^{2}\b S}\nonumber \\
 & \ge\norm{\mathcal{F}_{\G}\b{v_{k}+v_{k}^{\perp}}-\lambda\b{v_{k}+v_{k}^{\perp}}}_{L^{2}\b S}-\norm{\b{\F-\lambda}g}_{L^{2}\b S}\nonumber \\
 & =\norm{\mathcal{F}_{\G}\b{v_{k}+v_{k}^{\perp}}-\lambda\b{v_{k}+v_{k}^{\perp}}}_{L^{2}\b S}-\abs{\alpha_{k}-\lambda}\norm g_{L^{2}\b S},\label{eq:split}
\end{align}
where we use Eq. (\ref{eq:gassump-1}) in the last line. Also,
\[
\mathcal{F}_{\G}\b{v_{k}+v_{k}^{\perp}}-\lambda\b{v_{k}+v_{k}^{\perp}}=\b{\lambda_{k}-\lambda}v_{k}-\b{\mathcal{F}_{\G}-\lambda}v_{k}^{\perp},
\]
using the fact that $v_{k}$ is an eigenfunction.
We now need assumptions \ref{enu:coercive} and \ref{enu:angle},
noting that $v_{k}^{\perp}\in\mathcal{N}\b{\F_{\G}-\lambda_{k}}^{\perp}\cap C_{\G}^{q}\b S$.

\begin{align}
&\quad\norm{\mathcal{F}_{\G}\b{v_{k}+v_{k}^{\perp}}-\lambda\b{v_{k}+v_{k}^{\perp}}}_{L^{2}\b S}^{2}\nonumber\\
& =\norm{\b{\lambda_{k}-\lambda}v_{k}-\b{\F_{\G}-\lambda}v_{k}^{\perp}}_{L^{2}\b S}^{2}\nonumber \\
 & \ge\b{1-\theta_{k}}\b{\abs{\lambda-\lambda_{k}}^{2}\norm{v_{k}}_{L^{2}\b S}^{2}+\norm{\b{\F_{\G}-\lambda}v_{k}^{\perp}}_{L^{2}\b S}^{2}}\nonumber \\
 & \ge\b{1-\theta_{k}}\b{\abs{\lambda-\lambda_{k}}^{2}\norm{v_{k}}_{L^{2}\b S}^{2}+B_{k}^{2}\norm{v_{k}^{\perp}}_{L^{2}\b S}^{2}}\label{eq:thetabound}
\end{align}
Using (\ref{eq:gassump}), (\ref{eq:Guassump}), and (\ref{eq:gnorm}),
we have
\begin{equation}
\norm g_{L^{2}\b S}\le C_{k}\tilde{A}h_{\text{max}}^{p}\norm{\tilde{u}_{\lambda}}_{\H}.\label{eq:gl2bound}
\end{equation}
Then, using (\ref{eq:Fuassumpt}), (\ref{eq:split}), (\ref{eq:thetabound}),
and (\ref{eq:gl2bound}), 
\begin{align*}
\sqrt{\b{1-\theta_{k}}\b{\abs{\lambda-\lambda_{k}}^{2}\norm{v_{k}}_{L^{2}\b S}^{2}+B_{k}^{2}\norm{v_{k}^{\perp}}_{L^{2}\b S}^{2}}} & \le Ah_{\text{max}}^{p}\norm{\tilde{u}_{\lambda}}_{\H}+\abs{\alpha_{k}-\lambda}\norm g_{L^{2}\b S}\\
 & \le \b{A+\abs{\alpha_{k}-\lambda}C_{k}\tilde{A}}h_{\text{max}}^{p}\norm{\tilde{u}_{\lambda}}_{\H}\\
\implies\sqrt{\abs{\lambda-\lambda_{k}}^{2}\norm{v_{k}}_{L^{2}\b S}^{2}+B_{k}^{2}\norm{v_{k}^{\perp}}_{L^{2}\b S}^{2}} & \le\b{\tilde{C}_{k}+\tilde{\Lambda}_{k}\abs{\lambda-\lambda_{k}}}h_{\text{max}}^{p}\norm{\tilde{u}_{\lambda}}_{\H}\text{,}
\end{align*}
where $\ensuremath{\tilde{C}_{k},\tilde{\Lambda}_{k}>0}$ are constants.
So,
\begin{align}
\abs{\lambda-\lambda_{k}}\norm{v_{k}}_{L^{2}\b S} & \le\b{\tilde{C}_{k}+\tilde{\Lambda}_{k}\abs{\lambda-\lambda_{k}}}h_{\text{max}}^{p}\norm{\tilde{u}_{\lambda}}_{\H},\label{eq:valdif}\\
\norm{v_{k}^{\perp}}_{L^{2}\b S} & \le\frac{\b{\tilde{C}_{k}+\tilde{\Lambda}_{k}\abs{\lambda-\lambda_{k}}}}{B_{k}}h_{\text{max}}^{p}\norm{\tilde{u}_{\lambda}}_{\H},\nonumber \\
\norm{\tilde{u}_{\lambda}-v_{k}}_{L^{2}\b S} & \le\norm{v_{k}^{\perp}}_{L^{2}\b S}+\norm g_{L^{2}\b S}\nonumber \\
 & \le\b{\frac{\b{\tilde{C}_{k}+\tilde{\Lambda}_{k}\abs{\lambda-\lambda_{k}}}}{B_{k}}+C_{k}\tilde{A}}h_{\text{max}}^{p}\norm{\tilde{u}_{\lambda}}_{\H},\nonumber 
\end{align}
where we use (\ref{eq:gl2bound}) again in the line above. Letting
$\tilde{B}_{k}=\frac{\tilde{C}_{k}}{B_{k}}+C_{k}\tilde{A}$ and $\Lambda_{k}=\frac{\tilde{\Lambda}_{k}}{B_{k}}$,
\begin{align}
\norm{\tilde{u}_{\lambda}-v_{k}}_{L^{2}\b S} & \le\b{\tilde{B}_{k}+\Lambda_{k}\abs{\lambda-\lambda_{k}}}h_{\text{max}}^{p}\norm{\tilde{u}_{\lambda}}_{\H}\nonumber \\
\implies\norm{v_{k}}_{L^{2}\b S} & \ge\norm{\tilde{u}_{\lambda}}_{L^{2}\b S}-\b{\tilde{B}_{k}+\Lambda_{k}\abs{\lambda-\lambda_{k}}}h_{\text{max}}^{p}\norm{\tilde{u}_{\lambda}}_{\H}\nonumber \\
 & \ge\frac{\max\cb{\abs{b_{j}}}Q_{a}}{\norm{\tilde{u}_{\lambda}}_{\H}}-\b{\tilde{B}_{k}+\Lambda_{k}\abs{\lambda-\lambda_{k}}}h_{\text{max}}^{p}\norm{\tilde{u}_{\lambda}}_{\H},\label{eq:vbelow}
\end{align}
where we use (\ref{eq:normeq}) in the last line. Therefore, whenever
$\b{\tilde{B}_{k}+\Lambda_{k}\abs{\lambda-\lambda_{k}}}h_{\text{max}}^{p}\norm{\tilde{u}_{\lambda}}_{\H}^{2}<\max\cb{\abs{b_{j}}}Q_{a}$,
\begin{align*}
\frac{\norm{\tilde{u}_{\lambda}-v_{k}}_{L^{2}\b S}}{\norm{v_{k}}_{L^{2}\b S}} & \le\frac{\b{\tilde{B}_{k}+\Lambda_{k}\abs{\lambda-\lambda_{k}}}h_{\text{max}}^{p}\norm{\tilde{u}_{\lambda}}_{\H}^{2}}{\max\cb{\abs{b_{j}}}Q_{a}-\b{\tilde{B}_{k}+\Lambda_{k}\abs{\lambda-\lambda_{k}}}h_{\text{max}}^{p}\norm{\tilde{u}_{\lambda}}_{\H}^{2}},
\end{align*}
as required. Finally, using (\ref{eq:valdif}) and (\ref{eq:vbelow}),
\[
\abs{\lambda-\lambda_{k}}\le\frac{\b{\tilde{C}_{k}+\tilde{\Lambda}_{k}\abs{\lambda-\lambda_{k}}}h_{\text{max}}^{p}\norm{\tilde{u}_{\lambda}}_{\H}^{2}}{\max\cb{\abs{b_{j}}}Q_{a}-\b{\tilde{B}_{k}+\Lambda_{k}\abs{\lambda-\lambda_{k}}}h_{\text{max}}^{p}\norm{\tilde{u}_{\lambda}}_{\H}^{2}},
\]
as required.
\end{proof}

Note in the case that there are eigenvalues $\lambda_{j}$ both larger
and smaller than $\lambda_{k}$, then $\tilde{C}_{k}+\tilde{\Lambda}_{k}\abs{\lambda-\lambda_{k}}$
and $\tilde{B}_{k}+\Lambda_{k}\abs{\lambda-\lambda_{k}}$ could be
bounded above by constants due to assumption \ref{enu:lambda}.

\subsubsection{\label{subsec:Discussion-of-Propositions}Discussion of Proposition
\ref{prop:If-all-assumptions}}

Proposition \ref{prop:If-all-assumptions} should not be interpreted
as $\tilde{u}_{\lambda}\to v_{k}$ for any $\lambda$; the point is
that $\norm{\tilde{u}_{\lambda}}_{\H}$ is bounded only when there
is a solution to (\ref{eq:fulleigprob}) that is also in $\H$, by
Proposition \ref{prop:Let--be}. Proposition \ref{prop:If-all-assumptions}
is then better understood as a \textit{divergence rate} estimate for
$\norm{\tilde{u}_{\lambda}}_{\H}$ when $\lambda$ is not an eigenvalue
or as a convergence rate for intervals in $\lambda$ where $\norm{\tilde{u}_{\lambda}}_{\H}$
is below a certain constant; such intervals must shrink as $\O\b{h_{\text{max}}^{\frac{p}{2}}}$.
More explicitly, since

\[
\abs{\lambda-\lambda_{k}}\le\frac{\b{\tilde{C}_{k}+\tilde{\Lambda}_{k}\abs{\lambda-\lambda_{k}}}h_{\text{max}}^{p}\norm{\tilde{u}_{\lambda}}_{\H}^{2}}{\max\cb{\abs{b_{j}}}Q_{a}-\b{\tilde{B}_{k}+\Lambda_{k}\abs{\lambda-\lambda_{k}}}h_{\text{max}}^{p}\norm{\tilde{u}_{\lambda}}_{\H}^{2}},
\]
as long as $\max\cb{\abs{b_{j}}}Q_{a}>\b{\tilde{B}_{k}+\Lambda_{k}\abs{\lambda_{k}-\lambda}}h_{\text{max}}^{p}\norm{\tilde{u}_{\lambda}}_{\H}^{2}$, a rearrangement gives:

\begin{align*}
\frac{\max\cb{\abs{b_{j}}}Q_{a}\abs{\lambda-\lambda_{k}}h_{\text{max}}^{-p}}{\tilde{C}_{k}+\tilde{\Lambda}_{k}\abs{\lambda-\lambda_{k}}+\b{\tilde{B}_{k}+\Lambda_{k}\abs{\lambda-\lambda_{k}}}\abs{\lambda-\lambda_{k}}} & \le\norm{\tilde{u}_{\lambda}}_{\H}^{2}.
\end{align*}

Notably, for $\lambda$ sufficiently close to $\lambda_{k}$,
$\norm{\tilde{u}_{\lambda}}_{\H}^{2}$ is bounded below by a constant multiple of ${\abs{\lambda-\lambda_k}h_\text{max}^{-p}}$
as $h_{\text{max}},\abs{\lambda-\lambda_k}\to0$; $\norm{\tilde{u}_{\lambda}}_{\H}^{2}$ is  $\Omega\b{\abs{\lambda-\lambda_k}h_\text{max}^{-p}}$ in asymptotic notation.
If $\max\cb{\abs{b_{j}}}Q_{a}\le\b{\tilde{B}_{k}+\Lambda_{k}\abs{\lambda_{k}-\lambda}}h_{\text{max}}^{p}\norm{\tilde{u}_{\lambda}}_{\H}^{2}$,
 $\norm{\tilde{u}_{\lambda}}_{\H}^{2}$ is bounded below by a multiple of $h_\text{max}^{-p}$ as well. That is, in either case, $\norm{\tilde{u}_{\lambda}}_{\H}^{2}$ diverges as $h_\text{max}\to0$ as long as $\lambda\ne\lambda_k$. 
In the case that $\lambda=\lambda_{k}$, Proposition \ref{prop:If-all-assumptions}
says that $\norm{\tilde{u}_{\lambda_{k}}-v_{k}}_{L^{2}\b S}\to0$
as $\O\b{h_{\text{max}}^{p}}$, as long as there exists an extension
of an exact eigenfunction to $\H$ so that $\norm{\tilde{u}_{\lambda_{k}}}_{\H}$
is bounded (by the $\H$-norm of an exact, extended eigenfunction).

Now, recall that $\tilde{u}_{\lambda}^{\b{\tilde{N}}}$ is the solution
to problem (\ref{eq:opt1-1-1}) for $\tilde{N}$ points in $S$, and assume that $\tilde{N}_{\partial}\propto\tilde{N}^{\frac{m-1}{m}}$ and that $h_{\text{max}}=\O\b{\tilde{N}^{-\frac{1}{m}}}$, where
$m$ is the dimension of $S$ and $\tilde{N}_\partial$ is again the number of points on $\partial S$.
Then, the earlier Proposition \ref{prop:Let--be} says that
\begin{align}
\lim_{\tilde{N}_{2}\to\infty}\frac{\norm{\tilde{u}_{\lambda}^{\b{\tilde{N}_{2}}}}_{\H}}{\norm{u_{\lambda}^{\b{\tilde{N}_{1}}}}_{\H}} & =\infty,\label{eq:lim1}
\end{align}
for any $\tilde{N}_{1}$, when $\lambda$ is not an eigenvalue, and
\begin{align}
\lim_{\tilde{N}_{1}\to\infty}\b{\lim_{\tilde{N}_{2}\to\infty}\frac{\norm{\tilde{u}_{\lambda}^{\b{\tilde{N}_{2}}}}_{\H}}{\norm{\tilde{u}_{\lambda}^{\b{\tilde{N}_{1}}}}_{\H}}} & =1,\label{eq:lim2}
\end{align}
when $\lambda$ is an eigenvalue, as long as there is an eigenfunction for $\lambda$
that can be extended to some $u\in\H$. Limits (\ref{eq:lim1}) and (\ref{eq:lim2}) suggest that finding intervals where the norm ratio is less than some constant greater than one could be used to find intervals containing an eigenvalue. This is used for error estimation later in \ref{subsubsec:ee}.
Furthermore, using Proposition \ref{prop:If-all-assumptions},
\begin{equation}
\norm{\tilde{u}_{\lambda}^{\b{\tilde{N}}}}_{\H}\ge K_{k,\vc a}{\tilde{N}^{\frac{p}{2m}}\sqrt{\abs{\lambda-\lambda_{k}}}},\label{eq:divrate}
\end{equation}
for some constant $K_{k,\vc a}>0$, which gives the divergence rate of the first limit as $\tilde{N}\to\infty$.

It should also be noted that instead of using (\ref{eq:normeq}) to get Eq. (\ref{eq:vbelow}) in the proof of Proposition \ref{prop:If-all-assumptions}, we could use the fact that
\begin{align}\label{eq:Hlambdabound}
\norm{\tilde{u}_{\lambda} - \tilde{u}_{\lambda_k}}_\H \le Z_{k,\tilde{N}}\abs{\lambda - \lambda_k},
\end{align}
for some constant $Z_{k,\tilde{N}}>0$ depending on the point cloud, as long as $\abs{\lambda-\lambda_k}$ is sufficiently small.
This is fairly straightforward to show by writing the $\L$ operator from Subsection \ref{subsec:Functional-Analysis-Background} as $\L_k^{\b 0}+\b{\lambda -\lambda_k}\L_k^{\b 1}$, where $\L_k^{\b 0}$ and $\L_k^{\b 1}$ are independent of $\lambda$, then using Eq. (\ref{eq:kernel}) and noting that $\vc f$ does not depend on $\lambda$. Noting that for sufficiently small $\abs{\lambda - \lambda_k}$, $\norm{\b{\L_{k}^{\b 0}}^{+}-\b{\L_{k}^{\b 0}+\b{\lambda-\lambda_{k}}\L_{k}^{\b 1}}^{+}}=\O\b{\abs{\lambda-\lambda_{k}}}$ then gives (\ref{eq:Hlambdabound}). Assuming the $\H$ norm dominates the $L^2$ norm, $\norm{\tilde{u}_\lambda}_{L^2\b S} \ge \norm{\tilde{u}_{\lambda_k}}_{L^2\b S} - \tilde{Z}_{k,\tilde{N}}\abs{\lambda - \lambda_k}$ for some constant $\tilde{Z}_{k,\tilde{N}}>0$.
Using this, rather than (\ref{eq:normeq}), in the line before (\ref{eq:vbelow}) shows that when $\abs{\lambda-\lambda_k}$ is small and the point cloud is fixed, $\norm{\tilde{u}_\lambda}_{\H}$ rather than $\norm{\tilde{u}_\lambda}^2_{\H}$ is $\Omega\b{\abs{\lambda-\lambda_k}}$ as $\lambda\to\lambda_k$; this is the behaviour observed in practice.

Equations (\ref{eq:lim1})--(\ref{eq:divrate}) suggest that examining
$\norm{\tilde{u}_{\lambda}^{\b{\tilde{N}}}}_{\H}$ (or a ratio of
norms) can reveal eigenvalues. In particular, while $\norm{\tilde{u}_{\lambda}^{\b{\tilde{N}}}}_{\H}$
is bounded when $\lambda$ is an eigenvalue, it diverges at a potentially
high-order rate when $\lambda$ is not an eigenvalue. These observations
are used to find eigenvalues for a variety of problems in Section
\ref{sec:Application-to-Eigenvalue}.

\subsection{Analysis for Steklov Problems\label{subsec:Analysis-for-Steklov}}

Next, we analyze Steklov eigenvalue problems of the form

\begin{align}
\mathcal{F}u & =0\text{ on \ensuremath{S}},\label{eq:fulleigprob-1}\\
\b{\G u-\lambda u}\eval_{\partial S} & =0,\nonumber \\
u\b{\vc a_{j}} & =b_{j}\text{, for each \ensuremath{j\in\cb{1,2,\ldots,\tilde{N}_{a}}}},\nonumber 
\end{align}
where $\cb{\vc a_{j}}_{j=1}^{\tilde{N}_{a}}$ is a collection of points
on $\partial S$ and $\cb{b_{j}}_{n=1}^{\tilde{N}_{a}}$ is a set
of values such that at least one $b_{j}$ is non-zero. $\F$ and $\G$ are
again linear differential operators of orders $q$ and $\tilde{q}$,
respectively, such that $\tilde{q}\le q$.
Note that we must assume that $\vc a_{j}\in\partial S$
for this analysis. We define the Steklov eigenspace $E_{\lambda}$
to be the space of functions 
\[
E_{\lambda}:=\cb{u\in C^{q}\b{\overline{S}}:\F u=0,\text{\ensuremath{\b{\G u-\lambda u}\eval_{\partial S}}}=0}.
\]
We also define $E_{\lambda}\eval_{\partial S}\subset L^{2}\b{\partial S}$
to be the same space, restricted to the boundary of $S$. 

\subsubsection{Assumptions\label{subsec:Assumptions-1}}

We again need a handful of assumptions, as in Subsection \ref{subsec:Convergence-Rate-Analysis}.
Note that we will frequently omit $\evals_{\partial S}$ when the
meaning is otherwise clear; if $u$ is a function on $S$ that can
be continuously extended to $\partial S$, then $\norm u_{L^{2}\b{\partial S}}$
is taken to mean $\norm{u\evals_{\partial S}}_{L^{2}\b{\partial S}}$. Again, we note that these assumptions are satisfied by elliptic operators on sufficiently regular domains, such as those encountered in the numerical tests in Section \ref{sec:Application-to-Eigenvalue}.
\begin{enumerate}
\item \label{enu:stable_stek}There is a stable Robin (or Neumann) problem
for $\F$. Specifically, if $f\in L^{2}\b S$,
then there exist constants $B,\tilde{B}>0$ and $\alpha_{k}\in\C$ such that
if $\F u=f$ and $\b{\G-\alpha_{k}}u\eval_{\partial S}=0$, then $\norm u_{L^{2}\b{\partial S}}\le\tilde{B}\norm f_{L^{2}\b S}$. %
Also, there exists a linear extension map $\mathcal{E}:C^{\tilde{q}}\b{\partial S}\to L^{2}\b S\cap\mathcal{N}\b{\mathcal{F}}$
such that $\b{\G-\alpha_{k}}\b{\mathcal{E}g}\eval_{\partial S}=g$
for any $g\in C^{\tilde{q}}\b{\partial S}$. 
\item \label{enu:count_stek}There is a countable set of values $\cb{\lambda_{j}}$
(the Steklov eigenvalues) for which a solution $u\in C^{q}\b{\overline{S}}$
to (\ref{eq:fulleigprob-1}) exists.
\item \label{enu:closed_stek}For a specific $\lambda_{k}$ of interest,
the eigenspace $E_{\lambda_{k}}\eval_{\partial S}$ is closed as a
subspace of $L^{2}\b{\partial S}$ (finite-dimensional, for example).
\item \label{enu:lambda_stek}$\lambda$ is closer to $\lambda_{k}$ than
to any other Steklov eigenvalue in $\cb{\lambda_{j}}$.
\item \label{enu:coercive_stek}$\G-\lambda$ is uniformly coercive on $\b{E_{\lambda_{k}}\eval_{\partial S}}^{\perp}\cap\mathcal{N}\b{\F}\eval_{\partial S}$
for each $\lambda$ satisfying point \ref{enu:lambda_stek}, in the
sense that if $g\in C^{\tilde{q}}\b{\partial S}$, $g\perp E_{\lambda}\eval_{\partial S}$,
and there exists some $v^{\perp}\in\mathcal{N}\b{\F}$ so that $v^{\perp}\eval_{\partial S}=g$,
then there is some $C_{k}>0$ such that $\norm{\b{\G-\lambda}v^{\perp}}_{L^{2}\b{\partial S}}\ge C_{k}\norm{v^{\perp}}_{L^{2}\b{\partial S}}=C_{k}\norm g_{L^{2}\b{\partial S}}$.
This is the Steklov equivalent of assumption \ref{enu:coercive}
from \ref{subsec:Assumptions}. Note that $\mathcal{N}\b{\F}\eval_{\partial S}$
is equal to $L^{2}\b{\partial S}$ only when $\F u=0,u\eval_{\partial S}=g$
is solvable for any $g\in L^{2}\b{\partial S}$, which we do not require.
In particular, we want to consider the case $\F=-\Delta-\mu^{2}$,
where $\mu^{2}$ is a Dirichlet eigenvalue of the Laplacian. 
\item \label{enu:angle_stek}If $g\in\b{E_{\lambda_{k}}\eval_{\partial S}}^{\perp}\cap\mathcal{N}\b{\F}\eval_{\partial S}$,
then there is some $\theta_{k}\in\left[0,1\right)$ such that for
all $u\in E_{\lambda_{k}}$, $\abs{\b{\b{\G-\lambda}\b{\mathcal{E}g},u}_{L^{2}\b{\partial S}}}\ge\theta_{k}\norm{\b{\G-\lambda}\b{\mathcal{E}g}}_{L^{2}\b{\partial S}}\norm u_{L^{2}\b{\partial S}}$.
\end{enumerate}
~

For the discretized problem, we again consider dense point clouds
$\cb{\x_{j}}_{j=1}^{\infty}\subset S$ and $\cb{\vc y_{j}}_{j=1}^{\infty}\subset\partial S$
such that the fill distance $h_{\text{max}}$ of $\cb{\x_{j}}_{j=1}^{\tilde{N}}$
in $S$ and $\cb{\vc y_{j}}_{j=1}^{\tilde{N}_{\partial}}$ in $\partial S$
goes to zero as $\tilde{N},\tilde{N}_{\partial}\to\infty$. The discretized
Steklov problem is
\begin{align}
\text{minimize, over }u\in\H:\, & \norm u_{\H},\label{eq:opt1-1-1-1}\\
\text{subject to }\, & \mathcal{F}u\b{\x_{j}}=0\text{, for each \ensuremath{j\in\cb{1,2,\ldots,\tilde{N}}}},\nonumber \\
 & \mathcal{G}u\b{\vc y_{j}}-\lambda u\b{\vc y_{j}}=0\text{, for each \ensuremath{j\in\cb{1,2,\ldots,\tilde{N}_{\partial}}}},\nonumber \\
 & u\b{\vc a_{j}}=b_{j}\text{, for each \ensuremath{j\in\cb{1,2,\ldots,\tilde{N}_{a}}}}.\nonumber 
\end{align}

We assume $\H$ is such that there are constants $A,\tilde{A},p,Q_{a}>0$
such that the solution $\tilde{u}_{\lambda}$ to problem (\ref{eq:opt1-1-1-1})
satisfies,
\begin{align}
\norm{\mathcal{F}\tilde{u}_{\lambda}}_{L^{2}\b S} & \le Ah_{\text{max}}^{p}\norm{\tilde{u}_{\lambda}}_{\H},\label{eq:Fuassumpt-1}\\
\norm{\G\tilde{u}_{\lambda}-\lambda\tilde{u}_{\lambda}}_{L^{2}\b{\partial S}} & \le\tilde{A}h_{\text{max}}^{p}\norm{\tilde{u}_{\lambda}}_{\H},\label{eq:Guassump-1}\\
\norm{\tilde{u}_{\lambda}}_{L^{2}\b{\partial S}} & \ge\frac{\max\cb{\abs{b_{j}}}Q_{a}}{\norm{\tilde{u}_{\lambda}}_{\H}},\label{eq:normeq-1}
\end{align}
for sufficiently small $h_{\text{max}}$.

\subsubsection{Convergence-Divergence Result for Well-Posed Steklov Eigenvalue Problem}
\begin{proposition}
\label{prop:If-all-assumptions-1}If all assumptions from Subsection
\ref{subsec:Assumptions-1} hold, then there exist constants $K_{k},\tilde{K}_{k},\beta_{k},\tilde{\beta}_{k}>0$
such that for small enough $h_{\text{max}}$, there exists an eigenfunction
$v_{k}\in E_{\lambda}$ such that, if $\max\cb{\abs{b_{j}}}Q_{a}-\b{K_{k}+\beta_{k}\abs{\lambda-\lambda_{k}}}h_{\text{max}}^{p}\norm{\tilde{u}_{\lambda}}_{\H}^{2}$,
\begin{align*}
\abs{\lambda-\lambda_{k}} & \le\frac{\b{\tilde{K}_{k}+\tilde{\beta}_{k}\abs{\lambda-\lambda_{k}}}h_{\text{max}}^{p}\norm{\tilde{u}_{\lambda}}_{\H}^{2}}{\max\cb{\abs{b_{j}}}Q_{a}-\b{K_{k}+\beta_{k}\abs{\lambda-\lambda_{k}}}h_{\text{max}}^{p}\norm{\tilde{u}_{\lambda}}_{\H}^{2}},\\
\frac{\norm{\tilde{u}_{\lambda}-v_{k}}_{L^{2}\b{\partial S}}}{\norm{v_{k}}_{L^{2}\b{\partial S}}} & \le\frac{\b{K_{k}+\beta_{k}\abs{\lambda-\lambda_{k}}}h_{\text{max}}^{p}\norm{\tilde{u}_{\lambda}}_{\H}^{2}}{\max\cb{\abs{b_{j}}}Q_{a}-\b{K_{k}+\beta_{k}\abs{\lambda-\lambda_{k}}}h_{\text{max}}^{p}\norm{\tilde{u}_{\lambda}}_{\H}^{2}},
\end{align*}
where $\tilde{u}_{\lambda}$ is the solution to problem (\ref{eq:opt1-1-1-1}).
\end{proposition}

\begin{proof}
Using (\ref{eq:Fuassumpt-1}), we have
\[
\norm{\mathcal{F}\tilde{u}_{\lambda}}_{L^{2}\b S}\le Ah_{\text{max}}^{p}\norm{\tilde{u}_{\lambda}}_{\H}.
\]
Let $v=\mathcal{E}\b{\b{\G\tilde{u}_{\lambda}-\alpha_{k}\tilde{u}_{\lambda}}\eval_{\partial S}}$
so that $\b{\G\tilde{u}_{\lambda}-\alpha_{k}\tilde{u}_{\lambda}}\eval_{\partial S}=\b{\G v-\alpha_{k}v}\eval_{\partial S}$
and $\F v=0$ on $S$, where $\alpha_k$ is a constant that satisfies assumption \ref{enu:well-posed} from Subsection \ref{subsec:Assumptions-1}. Using assumption \ref{enu:well-posed}, we then have
\begin{align}
\norm{\tilde{u}_{\lambda}-v}_{L^{2}\b{\partial S}} & \le\tilde{B}Ah_{\text{max}}^{p}\norm{\tilde{u}_{\lambda}}_{\H}\label{eq:l2uvbound}\\
\implies\norm{\G\tilde{u}_{\lambda}-\G v}_{L^{2}\b{\partial S}} & \le\alpha_{k}\tilde{B}Ah_{\text{max}}^{p}\norm{\tilde{u}_{\lambda}}_{\H}\nonumber \\
\implies\norm{\b{\G-\lambda}\b{\tilde{u}_{\lambda}-v}}_{L^{2}\b{\partial S}} & \le\b{\abs{\alpha_{k}}+\abs{\lambda}}\tilde{B}Ah_{\text{max}}^{p}\norm{\tilde{u}_{\lambda}}_{\H}.\label{eq:g-lambound}
\end{align}
Now, using (\ref{eq:Guassump-1}),
\begin{align}
\norm{\G\tilde{u}_{\lambda}-\lambda\tilde{u}_{\lambda}}_{L^{2}\b{\partial S}} & \le\tilde{A}h_{\text{max}}^{p}\norm{\tilde{u}_{\lambda}}_{\H}\nonumber\\
\implies\norm{\G v-\lambda v}_{L^{2}\b{\partial S}} & \le\b{\tilde{A}+\b{\abs{\alpha_{k}}+\abs{\lambda}}\tilde{B}A}h_{\text{max}}^{p}\norm{\tilde{u}_{\lambda}}_{\H}\label{eq:glam},
\end{align}
where we used the bound on $\norm{\b{\G-\lambda}\b{\tilde{u}_{\lambda}-v}}_{L^{2}\b{\partial S}}$
from (\ref{eq:g-lambound}).

We now write $v=v_{k}+v_{k}^{\perp}$, where $v_{k}\in E_{\lambda_{k}}$.
This is done by projecting $v\eval_{\partial S}$ onto $E_{\lambda_{k}}\eval_{\partial S}$,
letting $v_{k}$ be a
function in $\mathcal{N}\b{\F}$ so that
$v_{k}\eval_{\partial S}$ is equal to this projection (recall that
$E_{\lambda_{k}}\subset\mathcal{N}\b{\F}$), then letting $v_{k}^{\perp}=v-v_{k}$.
Note that $v_{k}^{\perp}\in\mathcal{N}\b{\F}$ by construction and
$v_{k}^{\perp}\eval_{\partial S}\in\b{E_{\lambda_{k}}\eval_{\partial S}}^{\perp}$,
satisfying the conditions for assumption \ref{enu:coercive_stek}. Then, using (\ref{eq:glam}),
\begin{align*}
\norm{\G\b{v_{k}+v_{k}^{\perp}}-\lambda\b{v_{k}+v_{k}^{\perp}}}_{L^{2}\b{\partial S}} & \le\b{\tilde{A}+\b{\abs{\alpha_{k}}+\abs{\lambda}}\tilde{B}A}h_{\text{max}}^{p}\norm{\tilde{u}_{\lambda}}_{\H}.
\end{align*}
Therefore, since $\G v_{k}\eval_{\partial S}=\lambda_{k}v_{k}\eval_{\partial S}$,
\begin{align}
 & \quad\norm{\b{\lambda_{k}-\lambda}v_{k}+\b{\G-\lambda}v_{k}^{\perp}}_{L^{2}\b{\partial S}}\nonumber \\
 & \le\b{\tilde{A}+\b{\abs{\alpha_{k}}+\abs{\lambda}}\tilde{B}A}h_{\text{max}}^{p}\norm{\tilde{u}_{\lambda}}_{\H}\nonumber \\
\implies & \quad\sqrt{\b{1-\theta_{k}}\b{\abs{\lambda-\lambda_{k}}\norm{v_{k}}_{L^{2}\b{\partial S}}^{2}+C_{k}\norm{v_{k}^{\perp}}_{L^{2}\b{\partial S}}^{2}}}\label{eq:onemtheta}\\
 & \le\b{\tilde{A}+\b{\abs{\alpha_{k}}+\abs{\lambda}}\tilde{B}A}h_{\text{max}}^{p}\norm{\tilde{u}_{\lambda}}_{\H},\nonumber 
\end{align}
where we use assumptions \ref{enu:coercive_stek} and \ref{enu:angle_stek}
from Subsection \ref{subsec:Assumptions-1}. Finally, with new constants $K_{k},\beta_{k}>0$,
we use (\ref{eq:l2uvbound}) and (\ref{eq:onemtheta}) to see that
\begin{align}
\norm{\tilde{u}_{\lambda}-v_{k}}_{L^{2}\b{\partial S}} & \le\norm{v_{k}^{\perp}}_{L^{2}\b{\partial S}}+\norm{\tilde{u}_{\lambda}-v}_{L^{2}\b{\partial S}}\nonumber\\
&\le\b{K_{k}+\beta_{k}\abs{\lambda-\lambda_{k}}}h_{\text{max}}^{p}\norm{\tilde{u}_{\lambda}}_{\H}\label{eq:partialSbound} \\
\implies\norm{v_{k}}_{L^{2}\b{\partial S}} & \ge\norm{\tilde{u}_{\lambda}}_{L^{2}\b{\partial S}}-\b{K_{k}+\beta_{k}\abs{\lambda-\lambda_{k}}}h_{\text{max}}^{p}\norm{\tilde{u}_{\lambda}}_{\H}\nonumber \\
 & \ge\frac{\max\cb{\abs{b_{j}}}Q_{a}}{\norm{\tilde{u}_{\lambda}}_{\H}}-\b{K_{k}+\beta_{k}\abs{\lambda-\lambda_{k}}}h_{\text{max}}^{p}\norm{\tilde{u}_{\lambda}}_{\H}.\label{eq:vkbound}
\end{align}
Therefore, whenever $\max\cb{\abs{b_{j}}}Q_{a}\ge\b{K_{k}+\beta_{k}\abs{\lambda-\lambda_{k}}}h_{\text{max}}^{p}\norm{\tilde{u}_{\lambda}}_{\H}^{2}$,
\[
\frac{\norm{\tilde{u}_{\lambda}-v_{k}}_{L^{2}\b{\partial S}}}{\norm{v_{k}}_{L^{2}\b{\partial S}}}\le\frac{\b{K_{k}+\beta_{k}\abs{\lambda-\lambda_{k}}}h_{\text{max}}^{p}\norm{\tilde{u}_{\lambda}}_{\H}^{2}}{\max\cb{\abs{b_{j}}}Q_{a}-\b{K_{k}+\beta_{k}\abs{\lambda-\lambda_{k}}}h_{\text{max}}^{p}\norm{\tilde{u}_{\lambda}}_{\H}^{2}}.
\]
With other constants $\tilde{K}_{k}$, $\tilde{\beta}_{k}$, we
use (\ref{eq:onemtheta}) and (\ref{eq:vkbound}) to conclude
\[
\abs{\lambda-\lambda_{k}}\le\frac{\b{\tilde{K}_{k}+\tilde{\beta}_{k}\abs{\lambda-\lambda_{k}}}h_{\text{max}}^{p}\norm{\tilde{u}_{\lambda}}_{\H}^{2}}{\max\cb{\abs{b_{j}}}Q_{a}-\b{K_{k}+\beta_{k}\abs{\lambda-\lambda_{k}}}h_{\text{max}}^{p}\norm{\tilde{u}_{\lambda}}_{\H}^{2}}.
\]
\end{proof} 

Note that Eq. (\ref{eq:Guassump-1}) along with the bound on $\norm{\tilde{u}_{\lambda}-v_{k}}_{L^{2}\b{\partial S}}$ from (\ref{eq:partialSbound})
can be used to find a bound on $\norm{\tilde{u}_{\lambda}-v_{k}}_{L^{2}\b S}$
in the case that the $\G - \alpha_k$ Robin problem for $\F$ is suitably well-posed.

\section{Eigenvalue Problem Examples\label{sec:Application-to-Eigenvalue}}

\subsection{General Implementation}\label{subsec:genimpl}

We now consider problems of the form:
\begin{align}
\F u-\lambda\tilde{\F}u & =0\text{ on }S,\label{eq:eigpde}\\
\b{\G u-\lambda\tilde{\G}u}\eval_{\partial S} & =0.\nonumber 
\end{align}
where $\F,\G,\tilde{\F},\tilde{\G}$ are linear differential operators.
Let $\cb{\vc a_{j}}_{j=1}^{\tilde{N}_{a}}\subset S$ and $\cb{b_{j}}_{j=1}^{\tilde{N}_{a}}\in\R$
again (where at least one $b_{j}$ is non-zero), then the discretized
setup for these problems is
\begin{align}
\text{minimize, over }u\in\H: & \norm u_{\H},\label{eq:eigdisc}\\
\text{subject to } & \b{\F u-\lambda\tilde{\F}u}\b{\x_{j}}=0\text{, for each \ensuremath{j\in\cb{1,2,\ldots,\tilde{N}}}},\nonumber \\
 & \b{\G u-\lambda\tilde{\G}u}\b{\vc y_{j}}=0\text{, for each }j\in\cb{1,2,\ldots,\tilde{N}_{\partial}},\nonumber \\
 & u\b{\vc a_{j}}=b_{j}\text{, for each }j\in\cb{1,2,\ldots,\tilde{N}_{a}}.\nonumber 
\end{align}
When $\lambda\in\R$, the $\vc{\Phi}$ matrix for this problem, as
described in Subsection \ref{subsec:Direct-Solutions}, can be written
\[
\vc{\Phi}=\lambda^{2}\vc{\Phi}_{2}+\lambda\vc{\Phi}_{1}+\vc{\Phi}_{0},
\]
where $\vc{\Phi}_{2},\vc{\Phi}_{1}$, and $\vc{\Phi}_{0}$ are independent
of $\lambda$ and $\vc{\Phi}_{2},\vc{\Phi}_{0}$ are positive definite.

Given the solution $u_{\lambda}^{\b{\tilde{N}}}=\L^{*}\vc{\beta}$
to (\ref{eq:eigdisc}) (see Eq. (\ref{eq:finspan})), the linear
system for $\vc{\beta}$ can be written
\[
\vc{\Phi}\vc{\beta}=\vc g,
\]
where $\vc g$ has at most $\tilde{N}_{a}$ non-zero entries (and
at least one, since at least one of the $b_{j}$ values must be non-zero).
Recall that we can then compute
\[
\norm{u_{\lambda}^{\b{\tilde{N}}}}_{\H}^{2}=\vc{\beta}^{*}\vc{\Phi}\vc{\beta}=\vc g^{*}\vc{\beta}.
\]
We can also evaluate $\vc{\beta}^{*}\vc{\Phi}\vc{\beta}$ as

\begin{align}\label{eq:lamsplit}
\lambda^{2}\vc{\beta}^{*}\vc{\Phi}_{2}\vc{\beta}+\lambda\vc{\beta}^{*}\vc{\Phi}_{1}\vc{\beta}+\vc{\beta}^{*}\vc{\Phi}_{0}\vc{\beta},
\end{align}
and
\begin{align*}
\frac{\d}{\d\lambda}\norm{u_{\lambda}^{\b{\tilde{N}}}}_{\H}^{2} & =\vc g^{*}\frac{\d\vc{\Phi}^{-1}}{\d\lambda}\vc g
 =-\vc g^{*}\vc{\Phi}^{-1}\frac{\d\vc{\Phi}}{\d\lambda}\vc{\Phi}^{-1}\vc g
=-\vc g^{*}\vc{\Phi}^{-1}\b{2\lambda\vc{\Phi}_{2}+\vc{\Phi}_{1}}\vc{\Phi}^{-1}\vc g.
\end{align*}
Note that $\frac{\d\vc{\Phi}^{-1}}{\d\lambda}=-\vc{\Phi}^{-1}\frac{\d\vc{\Phi}}{\d\lambda}\vc{\Phi}^{-1}$
is simple to show since $\frac{\d}{\d\lambda}\b{\vc{\Phi}\vc{\Phi}^{-1}}=\vc 0$.
Then,
\[
\frac{\d^{2}}{\d\lambda^{2}}\norm{u_{\lambda}^{\b{\tilde{N}}}}_{\H}^{2}=2\vc g^{*}\vc{\Phi}^{-1}\frac{\d\vc{\Phi}}{\d\lambda}\vc{\Phi}^{-1}\frac{\d\vc{\Phi}}{\d\lambda}\vc{\Phi}^{-1}\vc g-2\vc g^{*}\vc{\Phi}^{-1}\vc{\Phi}_{2}\vc{\Phi}^{-1}\vc g.
\]
Since true eigenvalues should be located near local minima of $\norm{u_{\lambda}^{\b{\tilde{N}}}}_{\H}^{2}$,
this lets us use Newton's method applied to $\frac{\d}{\d\lambda}\norm{u_{\lambda}^{\b{\tilde{N}}}}_{\H}^{2}$ to find eigenvalues and lets us test concavity
to ensure that we are at a minimum rather than a maximum. As $\tilde{N}\to\infty$,
$\norm{u_{\lambda}^{\b{\tilde{N}}}}_{\H}^{2}$ will be unbounded except
at true eigenvalues for problem (\ref{eq:eigpde}) due to Proposition \ref{prop:Let--be}.

The main advantage of using the $\vc\Phi$ matrix to find eigenvalues is that the matrices $\vc\Phi_0,\vc\Phi_1$, and $\vc\Phi_2$ in Eq. (\ref{eq:lamsplit}) are independent of $\lambda$, and therefore must only be computed once before Newton's method or another minimization technique is used. If $\H$ is constructed from a truncated orthonormal basis as in Subsection \ref{subsec:Direct-Solutions}, this means that any computation with the truncated basis of size $N_b$ must only be done once, prior to minimizing $\norm{u_{\lambda}^{\b{\tilde{N}}}}_\H^2$. The computation time during any iterations of the optimization method does not scale with $N_b$, which is helpful since $N_b$ is typically much larger than the number of interpolation conditions. We use Newton's method with the $\vc\Phi$ matrix for the numerical tests in this section and are able to obtain accurate results. Where $\tilde{N}_{\text{total}}$ is the total number of interpolation conditions in (\ref{eq:eigdisc}), this approach involves $\O\b{\tilde{N}_{\text{total}}N_b}$ operations to form $\vc\Phi_0,\vc\Phi_1$, and $\vc\Phi_2$, then $\O\b{\tilde{N}_{\text{total}}^3}$ operations for each iteration of Newton's method.

An alternative is to evaluate $\norm{u_{\lambda}^{\b{\tilde{N}}}}_\H^2$ by solving (\ref{eq:eigdisc}) directly in the orthonormal basis using complete orthogonal decomposition (as in MATLAB's $\code{lsqminnorm}$) or singular value decomposition. The matrix form of this problem was stated in Eq. (\ref{eq:linalgopt}) in Subsection \ref{subsec:Direct-Solutions}. Minimization of $\norm{u_{\lambda}^{\b{\tilde{N}}}}_\H^2$ as a function of $\lambda$ can then be done using a derivative-free optimization method, such as MATLAB's $\code{fminbnd}$. This approach requires $\O\b{\tilde{N}_{\text{total}}^2N_b}$ operations to solve (\ref{eq:eigdisc}), and (\ref{eq:eigdisc}) must be solved at each iteration with different values of $\lambda$ when minimizing $\norm{u_{\lambda}^{\b{\tilde{N}}}}_\H^2$ as a function of $\lambda$. This is expensive but considerably more numerically stable than using $\vc\Phi$, as discussed at the end of Subsection \ref{subsec:Direct-Solutions}. 

Newton's method, using the expressions for the first and second derivatives of the squared $\H$-norm, is implemented in MATLAB for the tests in this section. %
Tests are run using a single thread of an Intel Core i5-8635U CPU and 16 GB of RAM. Code to run the tests in this section is available at \href{https://github.com/venndaniel/meshfree_eigenvalues}{https://github.com/venndaniel/meshfree\_eigenvalues}.

\subsection{Laplace--Beltrami on a Sphere\label{subsec:Laplace--Beltrami}}

As mentioned, meshing can be challenging for certain problems. This is often the case for surface
PDEs and provides motivation for the development of meshfree methods. Various techniques for geometry processing rely on the computation
of Laplace--Beltrami eigenvalues on surfaces \cite{reute06,nasik18}.
These techniques typically require a mesh. Reliable meshfree methods
to compute eigenvalues are rare in the literature; it is well known
that various “interpolate and differentiate” approaches to using
RBFs for PDEs do not have analytical guarantees for the eigenvalues
produced by the discretized operators. In particular, eigenvalues
can appear and remain on the wrong side of the half-plane, an observation
primarily made by those studying time-stepping methods using RBFs
(see the discussion at the end of Subsection 3.2 of \cite{alvar21}).

To alleviate this, approaches have been developed to ensure that eigenvalues
appear on the correct side of the half-plane for time-stepping, ranging
from a weak-form approach with Monte Carlo integration \cite{yan23},
to oversampling \cite{chen24}. Oversampling approaches have already
been proven to converge for solving elliptic problems with non-symmetric
RBF approaches \cite{cheun18}, whereas non-symmetric strong-form
RBF approaches with square systems (such as Kansa's original method
\cite{kansa90}) are known to potentially produce singular matrices
and not converge; the exact conditions under which such methods converge
is an open problem. 
Weak-form
approaches for eigenvalue and eigenfunction approximation with Monte Carlo
estimates of integrals (such as \cite{harli23}) converge slowly ($\O\b{\tilde{N}^{-\frac{1}{2}}}$
for eigenvalues) and cannot easily be used on surfaces with boundary.

For surface PDEs, we must consider a slightly different discretization
of problem (\ref{eq:eigpde}) to ensure that our surface differential
operators are computed correctly. For a concrete example, consider
a Laplace--Beltrami eigenvalue problem on a closed surface $S\subset\Omega$:
\begin{align}
-\Delta_{S}u-\lambda u & =0\text{ on }S\text{, }u\ne0.\label{eq:eiglapbeltrami}
\end{align}
We are not limited to closed surfaces; a numerical example for Steklov eigenvalues on a surface with boundary is shown in Subsection
\ref{subsec:Surface-Steklov}. However, we will start with a couple of examples on closed surfaces. Now, we recall a result from Xu and
Zhao (Lemma 1 of \cite{xu03}).
\\
\begin{lemma}\label{lem:lapbel}
Let $f\in C^{2}\b S$ and $\tilde{f}\in C^{2}\b{\Omega}$, where $S\subset\Omega$
is a thrice differentiable manifold of codimension one. If
\[
\tilde{f}\eval_{S}=f,
\]
then
\begin{align}
\Delta_{S}f=\Delta\tilde{f}-\kappa\hat{\vc n}_{S}\cdot\nabla\tilde{f}-\hat{\vc n}_{S}\cdot\b{D^{2}\tilde{f}}\hat{\vc n}_{S}\text{ on }S,\label{eq:lapbeldecomp}
\end{align}
where $\hat{\vc n}_{S}$ is the normal vector to $S$ and $\kappa=\nabla_{S}\cdot\hat{\vc n}_{S}$
is the sum of principal curvatures of $S$ (frequently called the
mean curvature in the computer graphics and numerical PDE communities).
\end{lemma}
~

The only term from Eq. (\ref{eq:lapbeldecomp}) that is not readily available when computing $\Delta_S u$ using an extended solution
$\tilde{u}$ on $\Omega$ to problem (\ref{eq:eiglapbeltrami}) is
$\kappa$. We therefore use the discretization:
\begin{align}
\text{minimize, over }u\in\H:\, & \norm u_{\H},\label{eq:eiglapbeltdisc}\\
\text{subject to }\, & \b{-\Delta u+\hat{\vc n}_{S}\cdot\b{D^{2}u}\hat{\vc n}_{S}-\lambda u}\b{\x_{j}}=0\text{, for \ensuremath{j\in\cb{1,2,\ldots,\tilde{N}}}},\nonumber \\
 & \hat{\vc n}_{S}\cdot\nabla u\b{\vc x_{j}}=0\text{, for }j\in\cb{1,2,\ldots,\tilde{N}},\nonumber \\
 & u\b{\vc a_{j}}=b_{j}\text{, for }j\in\cb{1,2,\ldots,\tilde{N}_{a}},\nonumber 
\end{align}
where $\cb{\x_{j}}_{j=1}^{\tilde{N}}=:S_{\tilde{N}}\subset S$ is
a point cloud, $\cb{\vc a_{j}}_{j=1}^{\tilde{N}_{a}}\subset S$ are
points where we ensure that $u$ is non-zero, and $\cb{b_{j}}_{j=1}^{\tilde{N}_{a}}$
are the non-zero values we must select.

For a simple test, we take $S$ to be the unit sphere. The true eigenvalues
of $-\Delta_{S}$ are $n\b{n+1}$ for non-negative integers $n$.
Our choice of $\H$, where $\cb{e^{i\vc{\omega}_{n}\cdot\vc x}}_{n=1}^{\infty}$
is a Fourier basis for $L^{2}$ functions on a box $\Omega=\sb{-\frac{\tilde\ell}{2},\frac{\tilde\ell}{2}}^{3}$
for some $\tilde\ell>0$, is
\begin{align}
\H & =\cb{u=\sum_{n=1}^{\infty}d_{n}^{-\frac{1}{2}}a_{n}e^{i\vc{\omega}_{n}\cdot\vc x}:a\in\ell^{2}},\label{eq:Hspace}
\end{align}
where
\begin{align*}
\b{\sum_{n=1}^{\infty}d_{n}^{-\frac{1}{2}}a_{n}e^{i\vc{\omega}_{n}\cdot\vc x},\sum_{n=1}^{\infty}d_{n}^{-\frac{1}{2}}b_{n}e^{i\vc{\omega}_{n}\cdot\vc x}}_{\H} & :=\b{a,b}_{\ell^{2}}.
\end{align*}
This is a space of Fourier extensions to the box $\Omega$ with smoothness
determined by $d$. Further discussion of these Hilbert spaces, their application in solving PDEs on surfaces, and their suitability for various domains is the subject of \cite{venn}, where the $\lambda=2$ eigenvalue is also computed numerically. Our choice of $d$, which will allow for super-algebraic
convergence, is
\begin{equation}
d_{n}=\exp\b{2q\b{\sqrt{\frac{2\pi}{T}}+\sqrt{\norm{\vc{\omega}_{n}}_{2}}}}.\label{eq:dnchoice}
\end{equation}
$T>0$ can be seen as an ``oscillation width'' of the functions
$\psi_{j}$; frequencies much greater than $\sqrt{\frac{2\pi}{T}}$
are heavily suppressed. With this choice of $d_{n}$, the norm in
$\H$ dominates all finite-order Sobolev norms $H^{p}\b{\Omega}$,
and all functions in $\H$ are $C^{\infty}$. We use $q=\tilde\ell=T=4$, $\tilde{N}_{a}=1$, $b_{1}=1$,
and $N_{b}=\b{2\cdot15+1}^{3}$ Fourier basis functions.

Our first test starts Newton's method, as described in Subsection \ref{subsec:genimpl}, at $\lambda=\b{\frac{n}{2}}^{2}$
for $n=\cb{0,1,2,\ldots30}$ to search for eigenvalues. $\tilde{N}$
scattered points are generated so that the average spacing between
neighbouring points is $h_{\text{avg}}=\O\b{\tilde{N}^{-\frac{1}{2}}}$.
A description of the point-generation algorithm is given in \ref{alg:boundary} (Version B), where we use $\tilde N_{\text{test},\partial}=40$. The algorithm ensures that the point spacing is closer to uniform than a
purely random point cloud, which improves conditioning and accuracy. 
We record the relative error between the relative minimum of $\norm{u_{\lambda}^{\b{\tilde{N}}}}_{\H}^{2}$
closest to $56$ as found by Newton's method using an absolute tolerance
of $10^{-8}$ and the true eigenvalue $\lambda=56$ (the 7th distinct, non-zero eigenvalue).
Note again that we expect the distance between the eigenvalue and a local
minimum of $\norm{u_{\lambda}^{\b{\tilde{N}}}}_{\H}^{2}$ to go to
zero as $\tilde{N}\to\infty$ at a high-order rate (see Proposition
\ref{prop:If-all-assumptions} and the following discussion in Subsection
\ref{subsec:Discussion-of-Propositions}). A plot of the relative
error against $\tilde{N}$ for the $\lambda=56$ eigenvalue is displayed
in Fig. \ref{fig:Relative-error-for}.

\begin{figure}[H]
\begin{centering}
\includegraphics[width=119mm]{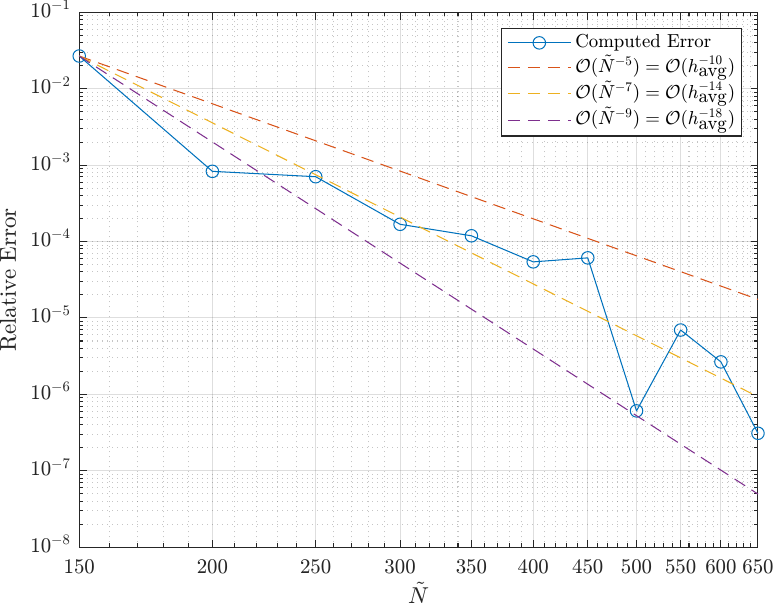}
\par\end{centering}
\caption{\label{fig:Relative-error-for}Relative error for the nearest computed
Laplace--Beltrami eigenvalue to $\lambda=56$ using $\tilde{N}$ scattered points on the unit sphere. $h_{\text{avg}}$ is the average
point spacing, which scales as $\protect\O\protect\b{\tilde{N}^{-\frac{1}{2}}}$
on a 2D surface}
\end{figure}

Note that the point generation is partially random, so we expect some
irregularities in the convergence results, as seen in Fig. \ref{fig:Relative-error-for}.
The overall convergence appears to be high-order.
For the $\tilde{N}=650$ test, we list the first 14 computed non-zero
eigenvalues (not including multiplicities) along with their
true values in Table \ref{tab:First-15-eigenvalues}. This test takes 34 seconds to run after the matrices $\vc\Phi_0, \vc \Phi_1$, and $\vc\Phi_2$ are formed; Newton's method for the individual eigenvalue $\lambda=56$ takes 2.3 seconds. Matrix formation takes 8 seconds for $\tilde{N} = 650$. The computed
eigenvalue corresponding to the true eigenvalue $\lambda=0$ was $10^{-8}$;
Newton's method was run to a tolerance of $10^{-8}$. Note that $\lambda=210$ is the 197th to 225th eigenvalue when multiplicities
are included.

\begin{table}[h]
\caption{First 14 non-zero eigenvalues on the unit sphere computed
via minimization of $\protect\norm{u_{\lambda}^{\protect\b{\tilde{N}}}}_{\protect\H}^{2}$
with $\tilde{N}=650$ compared with the true eigenvalues and the corresponding
relative error}\label{tab:First-15-eigenvalues}
\begin{tabular}{@{}lllllll@{}}
\toprule 
Computed $\lambda$ & True $\lambda$ & Relative Error && Computed $\lambda$ & True $\lambda$ & Relative Error\\
\cmidrule{1-3}\cmidrule{5-7}

1.99999999 & 2 & 5.7728E-09 && 71.99981050 & 72 & 2.6319E-06\\

5.99999996 & 6 & 7.3495E-09&&89.99979568 & 90 & 2.2702E-06\\

11.99999973 & 12 & 2.2910E-08&&109.99830298 & 110 & 1.5427E-05\\

19.99999965 & 20 & 1.7479E-08&&131.99507486 & 132 & 3.7312E-05\\

29.99999981 & 30 & 6.3053E-09&&155.97416338 & 156 & 1.6562E-04\\

41.99998409 & 42 & 3.7875E-07&&181.88803521 & 182 & 6.1519E-04\\

55.99998263 & 56 & 3.1010E-07&&209.75390081 & 210 & 1.1719E-03\\
\bottomrule
\end{tabular}
\end{table}
 
These tests show that knowing $\norm{u_{\lambda}^{\b{\tilde{N}}}}_{\H}^{2}$
is bounded if and only if $\lambda$ is an eigenvalue is sufficient
for computing Laplace--Beltrami eigenvalues on an unstructured
point cloud. Furthermore, estimates for the smaller eigenvalues are
highly accurate, even for a fairly large point spacing ($\tilde{N}=650$
roughly corresponds to a distance on the order of $10^{-1}$ between
points, on average).

\subsection{Laplace--Beltrami on a Genus 2 Surface\label{subsec:Genus2}}

Implicitly-defined surfaces generally do not have a parametrization available and may be difficult or expensive to mesh. Conversely, forming a point cloud on an implicit surface, as needed for our method, is relatively straightforward.
For an example of a more complicated surface, we take a genus 2 surface defined as the zero set of the function:
\[
\varphi\b{x,y,z} := \frac{1}{4\b{\b{x-1}^{2}+y^{2}}}+\frac{1}{4\b{\b{x+1}^{2}+y^{2}}}+\frac{1}{10}x^{2}+\frac{1}{4}y^{2}+z^{2}-1.
\] 
The point cloud is again generated using the algorithm in \ref{alg:boundary} (Version B) with $\tilde{N}_{\text{test},\partial}=40$. To better understand the method, it is helpful to examine a plot of $\norm{u_{\lambda}^{\b{\tilde{N}}}}_{\H}$ as a function of $\lambda$; this is shown in Fig. \ref{fig:hnorm}. We use a Hilbert space given by Eq. (\ref{eq:Hspace}) and a choice of $d_n$ given by Eq. (\ref{eq:dnchoice}) with parameters $q=5$ and $T=12$. $N_b=\b{2\cdot15+1}^3$ Fourier basis functions on the box $\Omega=\sb{-5,5}\times\sb{-3,3}\times\sb{-1.5,1.5}$ are used for the tests in this subsection.

\begin{figure}[H]
\begin{centering}
\includegraphics[width=119mm]{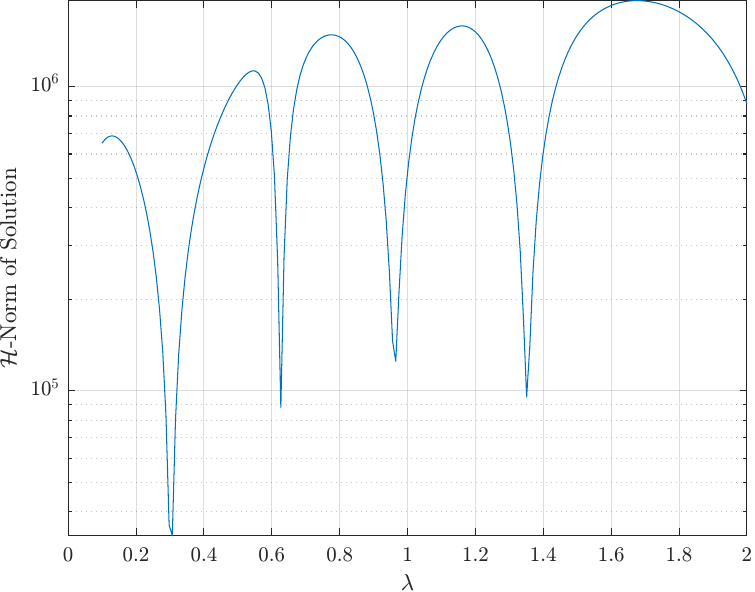}
\par\end{centering}
\caption{\label{fig:hnorm}$\norm{u_{\lambda}^{\b{\tilde{N}}}}_{\H}$ as a function of $\lambda$ for $\tilde{N}=1200$ for a genus 2 surface}
\end{figure}

Fig. \ref{fig:hnorm} shows four clear minima of $\norm{u_{\lambda}^{\b{\tilde{N}}}}_{\H}$ for $\lambda\in\b{0.1,2}$; these should correspond to eigenvalues of the Laplace--Beltrami operator on this surface. We test the convergence of the minima near $\lambda=0.3$ in Table \ref{tab:firstgenus2}, computed using Newton's method initialized at 0.3 and run for 20 iterations or to an absolute tolerance of $10^{-12}$. We conclude that the first non-zero eigenvalue is around $0.3025$ or $0.3026$. Table \ref{tab:secondgenus2} gives estimates of the second non-zero eigenvalue with Newton's method initialized at the rough estimate of $\lambda=0.63$; the method produces a result around 0.6263 or 0.6264. In both tests, we observe convergence of the first few digits with point spacing on the order of $10^{-1}$. The eigenfunction corresponding to $\lambda\approx0.626$ is shown in Fig. \ref{fig:eigfuncg2}.

\begin{table}[h]
\caption{First non-zero eigenvalue estimate for a genus two surface, with computation times}\label{tab:firstgenus2}
\begin{tabular}{@{}lllll@{}}
\toprule 
$\tilde{N}$ & Computed $\lambda$ & Relative Change & Time ($\vc \Phi_j$ Matrices) & Time (Newton's Method)\\
\midrule

400 & 0.3098406 & N/A & 3.31 s & 0.28 s\tabularnewline

800 & 0.3022492 & 2.5116E-02 & 13.98 s & 1.60 s\tabularnewline

1200 & 0.3025654 & 1.0452E-03 & 33.18 s & 10.29 s\tabularnewline

1600 & 0.3025519 & 4.4573E-05 & 53.81 s & 16.86 s\tabularnewline
\bottomrule
\end{tabular}

\end{table}

\begin{table}[h]
\caption{Second non-zero eigenvalue estimate for a genus two surface, with computation times, noting that the $\vc \Phi _j$ matrices do not need to be computed again after the test of Table \ref{tab:firstgenus2}}\label{tab:secondgenus2}
\begin{tabular}{@{}llll@{}}
\toprule 
$\tilde{N}$ & Computed $\lambda$ & Relative Change & Time (Newton's Method)\\
\midrule

800 & 0.6274113 & N/A & 3.82 s\tabularnewline

1200 & 0.6264340 & 1.5601E-03 & 8.43 s\tabularnewline

1600 & 0.6263515 & 1.3165E-04 & 18.17 s\tabularnewline

\bottomrule
\end{tabular}

\end{table}

\begin{figure}[ht]
\begin{centering}
\includegraphics[width=119mm]{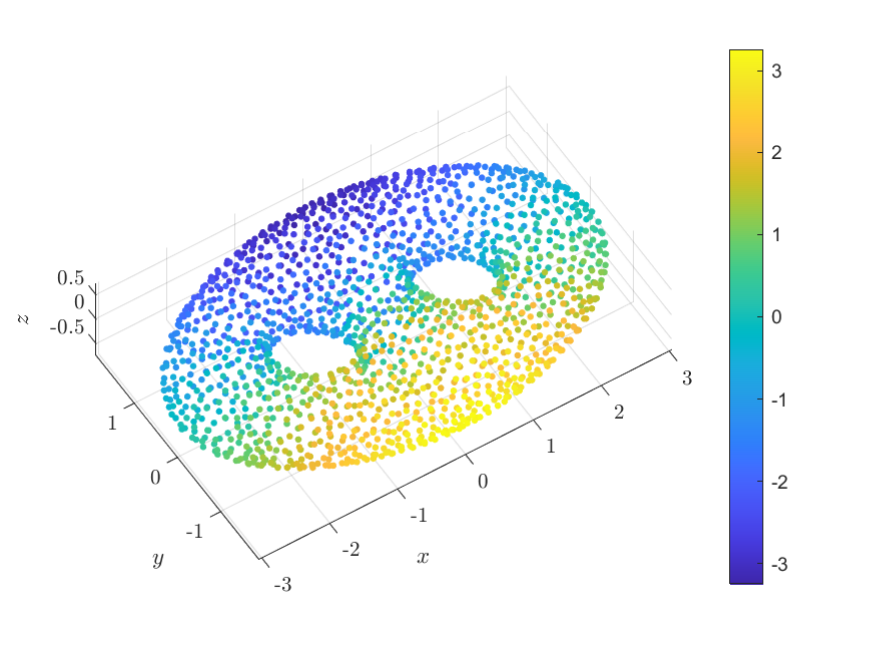}
\par\end{centering}
\caption{\label{fig:eigfuncg2}The Laplace--Beltrami eigenfunction corresponding to $\lambda\approx0.626$ for a genus 2 surface, computed with $\tilde{N}=1600$ points}
\end{figure}

An alternative approach, for when the mean curvature of the surface is readily available, is to include the first normal derivative term in Laplace--Beltrami computation directly. That is, instead of using Eq. (\ref{eq:eiglapbeltdisc}), we instead solve the following optimization problem:

\begin{align}
\text{minimize, over }u\in\H:\, & \norm u_{\H},\label{eq:includekappa}\\
\text{subject to }\, & \b{-\Delta u+\hat{\vc n}_{S}\cdot\b{D^{2}u}\hat{\vc n}_{S} + \kappa \hat{\vc n}_{S}\cdot\nabla u\b{\vc x_{j}}-\lambda u}\b{\x_{j}}\nonumber\\
&\quad=0\text{, for \ensuremath{j\in\cb{1,2,\ldots,\tilde{N}}}},\nonumber \\
 & u\b{\vc a_{j}}=b_{j}\text{, for }j\in\cb{1,2,\ldots,\tilde{N}_{a}},\nonumber 
\end{align}
where $\kappa=\nabla\cdot\frac{\nabla\varphi}{\norm{\nabla\varphi}_2}$. This approach uses Lemma \ref{lem:lapbel} more directly and is numerically advantageous since it imposes fewer conditions. However, it is only applicable when $\kappa$ can be computed. In this case, $\kappa$ is known since we have a level set for the surface, but it is possible to estimate $\kappa$ accurately using only point cloud data via meshfree interpolation of a (local) level set. Methods for fitting level sets through unorganized points (thus computing normal vectors and mean curvature) are prevalent in the computer graphics community (see \cite{carr01}, for example); we take normal vectors as given throughout this paper and mean curvature as given for the rest of this subsection.
Estimates of the first two non-zero eigenvalues using the approach from (\ref{eq:includekappa}) are given in Tables \ref{tab:firstgenus2k} and \ref{tab:secondgenus2k}. We note considerable improvement. Computation time is also significantly reduced since the matrices involved are smaller. Note again that the $\vc \Phi_j$ matrices only need to be computed once per point cloud.

\begin{table}[h]
\caption{First non-zero eigenvalue estimate for a genus two surface using the modified approach given by Eq. (\ref{eq:includekappa}), with computation times}\label{tab:firstgenus2k}
\begin{tabular}{@{}lllll@{}}
\toprule 
$\tilde{N}$ & Computed $\lambda$ & Relative Change & Time ($\vc \Phi_j$ Matrices) & Time (Newton's Method)\\
\midrule

400 & 0.3019843 & N/A & 1.69 s & 0.06 s\tabularnewline

800 & 0.3025272 & 1.7948E-03 & 4.33 s & 0.24 s\tabularnewline

1200 & 0.3025111 & 5.3416E-05 & 8.91 s & 0.44 s\tabularnewline

1600 & 0.3025207 & 3.1804E-05 & 11.89 s & 3.56 s\tabularnewline

2000 & 0.3025205 & 7.2296E-07 & 29.68 s & 5.76 s\\

\bottomrule
\end{tabular}

\end{table}

\begin{table}[h]
\caption{Second non-zero eigenvalue estimate for a genus two surface using the modified approach given by Eq. (\ref{eq:includekappa}), with computation times, noting that the $\vc \Phi _j$ do not need to be computed again after the test of Table \ref{tab:firstgenus2k}}\label{tab:secondgenus2k}
\begin{tabular}{@{}lllll@{}}
\toprule 
$\tilde{N}$ & Computed $\lambda$ & Relative Change & Time (Newton's Method)\\
\midrule

400	&0.6265373&N/A	& 0.06 s\\
800	&0.6262142	&5.1589E-04 & 0.19 s\\
1200	&0.6263335	&1.9035E-04 & 1.45 s\\
1600	&0.6263370	&5.7036E-06 & 0.73 s\\
2000	&0.6263408	&6.0671E-06 & 5.68 s\\

\bottomrule
\end{tabular}

\end{table}

\subsection{Laplacian on an Irregular 2D Domain}\label{subsec:irreg}

We now look at a 2D problem on an irregularly-shaped domain with a hole. The domain is defined by the zero set of 
\[
\vp\b{x,y} := x^2 + y^2 - 0.2\sin\b{5x}-0.2\sin\b{6y} + \frac{0.1}{x^2 + y^2} -1.
\]
A point cloud for this domain, along with a computed eigenfunction of the Laplacian with Dirichlet boundary conditions, is shown in Figure \ref{fig:lapeigfunc}. We compare the meshfree method against P2 finite elements on a mesh generated by a Delaunay triangulation of the point cloud, post-processed to remove triangles outside the domain. We use the space $\H$ from Subsection \ref{subsec:Laplace--Beltrami}, but with $\vc \omega _n\in\R^2$. The numerical discretization is

\begin{align}\label{eq:irregdomainopt}
\text{minimize, over }u\in\H: & \norm u_{\H},\\
\text{subject to } & -\b{\Delta u}\b{\x_{j}} -\lambda u\b{\x _j}=0\text{, for each \ensuremath{j\in\cb{1,2,\ldots,\tilde{N}}}},\nonumber \\
 & u\b{\vc y_{j}}=0\text{, for each }j\in\cb{1,2,\ldots,\tilde{N}_\partial},\nonumber\\
 & u\b{\vc a_{j}}=b_j\text{, for each }j\in\cb{1,2,\ldots,\tilde{N}_{a}}.\nonumber
\end{align}

Table \ref{tab:2dlap} shows the first estimated eigenvalue using the meshfree method with $q=2.5,T=0.5,N_b=\b{2\cdot75 + 1}^2,$ and $\Omega = \sb{-1.5,1.5}^2$ compared to the eigenvalues from a P2 finite element method (FEM P2), computed using FreeFEM++ \cite{freefem}. $\tilde{N}\approx \tilde{N}_\partial^2/30$ interior points are generated using the algorithm in \ref{alg:interior} with $w=0$, while the boundary points are generated with \ref{alg:boundary} (Version B). The recorded time for the meshfree method is for the Newton's method iterations after forming the $\vc \Phi_j$ matrices, while the FEM times are for finding the first non-zero eigenvalue after the mesh and matrices are already formed. Newton's method is initialized at $\lambda=28.5$ and is run with an absolute tolerance of $10^{-5}$.
   
\begin{table}[h]

\caption{\label{tab:2dlap}First Laplacian eigenvalue on a 2D domain with a hole
computed via minimization of $\protect\norm{u_{\lambda}^{\protect\b{\tilde{N}}}}_{\protect\H}^{2}$
(meshfree, left) compared to a FEM computation with P2 elements (right). An additional computation with $\tilde N_\partial=640$ is given for the FEM computation only for comparison to the final meshfree values.
}

\begin{tabular}{@{}lllllll@{}}
\toprule 
$\tilde{N}_\partial$ & Meshfree $\lambda$ & Relative Change & Time& FEM P2 $\lambda$ & Relative Change & Time\\
\midrule

120 & 28.46418 & N/A & 0.2 s&28.38149 & N/A & 0.1 s\tabularnewline

160 & 28.48687 & 7.9676E-04 & 0.5 s&28.45701 & 2.6539E-03 & 0.1 s\tabularnewline

200 & 28.48522 & 5.7903E-05 & 1.3 s&28.45443 & 9.0844E-05 & 0.2 s\tabularnewline

240 & 28.48403 & 4.1963E-05 & 2.2 s&28.45910 & 1.6409E-04 & 0.4 s\tabularnewline

280 & 28.48408 & 1.6729E-06 & 3.9 s&28.46771 & 3.0248E-04 & 0.5 s\tabularnewline

320 & 28.48409 & 5.2572E-07 & 8.8 s&28.47623 & 2.9921E-04 & 0.7 s\tabularnewline
\tabularnewline
640 &  &  & &28.48212 & 2.0681E-04 & 2.9 s\\

\bottomrule
\end{tabular}

\end{table}

We see that the meshfree method consistently obtains the first 5 digits of the eigenvalue starting at $\tilde{N}_\partial=240$. At $\tilde{N}_\partial=640$, a standard finite element approach does not yet seem to obtain the fifth digit. The $\tilde{N}_\partial=640$ FEM test demonstrates that the result with a standard method does indeed appear to be converging to the value found with the meshfree method. Examining computation times, we see that while the meshfree method takes longer for the same number of points, as expected due to the dense matrices involved, higher accuracy is generally achieved for the same computation time. This becomes more pronounced for higher computation times; the $\tilde{N}_\partial=240$ test for the meshfree method takes 2.2 s but appears to produce a considerably more accurate result than the $\tilde{N}_\partial=640$ FEM test, which takes 2.9 s. We reiterate that these times are for computing a single eigenvalue after the necessary setup has already been done. This setup involves forming the $\vc\Phi_j$ matrices for the meshfree method, and meshing and matrix formation for the FEM. Meshing an arbitrary level set, such as the one in this test, without a parametrization is not a feature included in many standard FEM packages. This could be time-consuming for the mesh resolutions needed to achieve a similar accuracy as the final tests of the meshfree method. %
The scaled eigenfunction for $\tilde{N}_\partial=320$ computed with the meshfree method is shown in Figure \ref{fig:lapeigfunc}. Overall, we conclude that the meshfree method can accurately compute eigenvalues on irregularly-shaped domains with holes.

\begin{figure}[H]
\begin{centering}
\includegraphics[width=119mm]{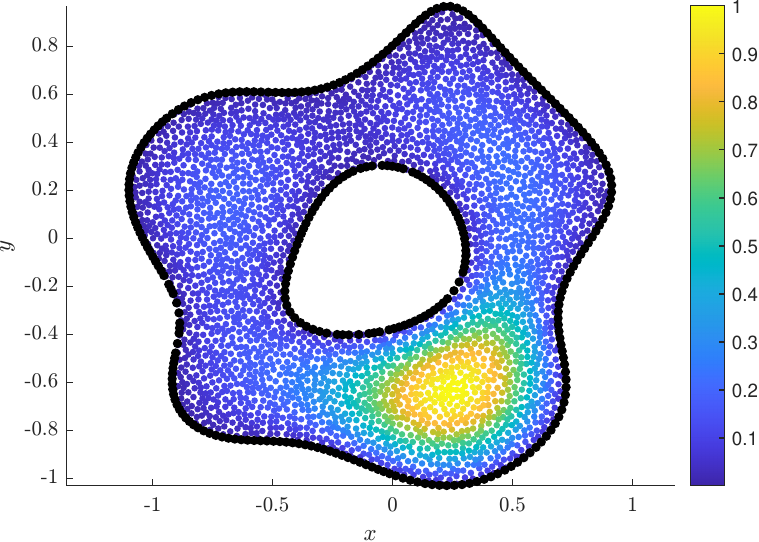}
\par\end{centering}
\caption{\label{fig:lapeigfunc}Eigenfunction for $\lambda\approx 28.48409$ computed via the meshfree method with $\tilde{N}_\partial=320$ boundary points and $\tilde{N}=3413$ interior points}
\end{figure}

\subsection{Steklov Eigenvalues\label{subsec:Steklov-Eigenvalues}}

For a domain with boundary $S$, the Steklov problem (for the Laplacian)
is given by
\begin{align*}
\Delta u & =0\text{ on }S,\\
\b{\hat{\vc n}\cdot\nabla u-\lambda u}\eval_{\partial S} & =0.
\end{align*}
That is, the eigenvalue $\lambda$ is in the boundary condition, not
in the interior of the domain. A variety of methods have been proposed
for this problem; primarily, these are finite element \cite{han15,li13,weng15,xiong23,you19}
or boundary integral methods \cite{akhme17,chen20b}. Adapting this
problem for surfaces, which we do later in Subsection \ref{subsec:Surface-Steklov},
is typically difficult. This is particularly true for boundary integral methods,
since Green's functions for the surface may not be known or feasible to compute, except for certain simple surfaces such as the sphere \cite{gemmr08}.
Boundary integral methods also rely on the availability of a suitable
quadrature scheme, which also might not be available for surfaces.
Also, as usual, meshing is more difficult on surfaces than for flat
domains. A meshfree method for this problem could therefore be quite
useful, especially on surfaces.

A major advantage of our approach is that it is completely universal
for linear PDEs; there are no additional complications from a Steklov
problem compared to the usual Laplace eigenvalue problem. We consider:

\begin{align}
\text{minimize, over }u\in\H: & \norm u_{\H},\label{eq:eigsteklov}\\
\text{subject to } & \b{\Delta u}\b{\x_{j}}=0\text{, for each \ensuremath{j\in\cb{1,2,\ldots,\tilde{N}}}},\nonumber \\
 & \hat{\vc n}_{S}\cdot\nabla u\b{\vc y_{j}}-u\b{\vc y_{j}}=0\text{, for each }j\in\cb{1,2,\ldots,\tilde{N}_{\partial}},\nonumber \\
 & u\b{\vc a_{j}}=1\text{, for each }j\in\cb{1,2,\ldots,\tilde{N}_{a}}.\nonumber 
\end{align}

As a test example, we examine this problem for the unit disk. The
true Steklov eigenvalues for this problem are known to be the non-negative
integers (see, for example, Example 1.3.1. of \cite{girou17}).
A note for the Steklov problem is that solutions are
known to decay rapidly away from the boundary (see Thm. 1.1 of \cite{hislo01}); for this reason, it 
makes sense to place the $\vc a_{j}$ points on the boundary. This
approach is more successful numerically than placing $\vc a_{j}$
points near the centre and is supported by the analysis of Subsection
\ref{subsec:Analysis-for-Steklov}, which requires $\vc a_{j}$ to
be on the boundary. We set $u\b{\vc a_j}=1$ on one point on the boundary ($\tilde{N}_{a}=1$), and use
$N_{b}=\b{2\cdot75+1}^{2}$, $\Omega=\sb{-2,2}^{2}$, $q=4$, and $T=1$,
with $\H$ as in Subsection \ref{subsec:Laplace--Beltrami},
so that $d_{n}$ is given by Eq. (\ref{eq:dnchoice}) (where
$\vc{\omega}_{n}$ is now in $\R^{2}$ rather than $\R^{3}$). We note here that the basis functions used to construct $\H$ are not radial and are not related to the eigenfunctions of the unit disk; $\H$ could instead be constructed using radial functions, which would likely yield more rapid convergence. However, the point of this test is partially to show that we can use a standard Fourier basis on the box to accurately find eigenvalues on a non-box domain. 

We
take $\tilde{N}\approx\b{\frac{\tilde{N}_{\partial}}{4}}^{2}$, with
more scattered points near the boundary than in the middle of the
domain; this is done through a similar process as the point cloud
in the previous subsection, but by weighting distances from the existing
point cloud by $4\b{1-r^{2}}+1$, where $r$ is the distance of a
potential new point from the origin. This gives a preference for points
near the boundary, which we have observed numerically to improve
convergence. The specific processes used are given in \ref{alg:boundary} (Version A) for the $\tilde{N}_\partial$ points on the boundary and \ref{alg:interior} for the $\tilde{N}$ points in the interior, with $w=4$. Notably, the generated point cloud will not have any symmetry properties. Again, a choice of point cloud that is better tailored to the domain or $\H$ may be possible, but this does not generalize well to irregular domains or surfaces. Therefore, we use the more general approach for point cloud generation here as a demonstration, despite the simple domain.
As an initial test, we search for the $\lambda=10$ Steklov eigenvalue
at various resolutions. The result is shown in Fig. \ref{fig:Relative-error-for-1}. A absolute tolerance of $10^{-8}$ is used as a stopping criterion for Newton's method.

\begin{figure}[H]
\begin{centering}
\includegraphics[width=119mm]{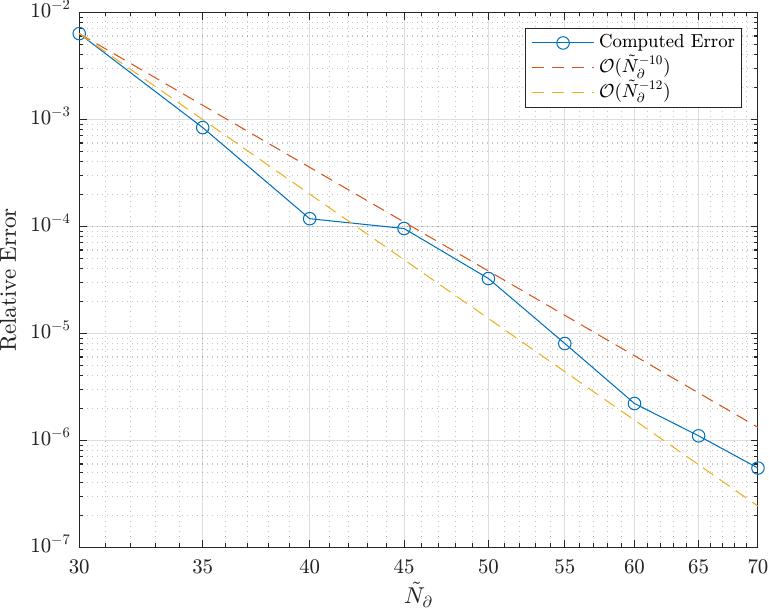}
\par\end{centering}
\caption{\label{fig:Relative-error-for-1}Relative error for the nearest computed
Steklov eigenvalue to $\lambda=10$ using $\tilde{N}_{\partial}$
scattered points on the boundary of the unit disk and $\approx\tilde{N}_{\partial}/4$
points in the interior}
\end{figure}

Also, with $\tilde{N}_{\partial}=65$, we list the first 20 computed
Steklov eigenvalues in Table \ref{tab:First-10-eigenvalues}, not
including multiplicities. We capture the first 20 eigenvalues fairly accurately (39 including
multiplicity); the behaviour is similar to the Laplace--Beltrami
test from Subsection \ref{subsec:Laplace--Beltrami}, demonstrating the generality of
the method.

\begin{table}[h]

\caption{\label{tab:First-10-eigenvalues}First 20 Steklov eigenvalues on the unit disk computed via
minimization of $\protect\norm{u_{\lambda}^{\protect\b{\tilde{N}}}}_{\protect\H}^{2}$
with $\tilde{N}_{\partial}=65$ compared with the true eigenvalues
and the corresponding relative error}
\begin{tabular}{@{}lllllll@{}}
\toprule 
Computed $\lambda$ & True $\lambda$ & Relative Error&&Computed $\lambda$ & True $\lambda$ & Relative Error\tabularnewline
\cmidrule{1-3}\cmidrule{5-7}

0.00000003	&	0	&	N/A	&&10.00001105	&	10	&	1.1048E-06	\tabularnewline
1.00000006	&	1	&	6.3018E-08	&&11.00002931	&	11	&	2.6646E-06	\tabularnewline
2.00000016	&	2	&	8.0038E-08	&&12.00004538	&	12	&	3.7814E-06	\tabularnewline
3.00000017	&	3	&	5.6218E-08	&&13.00010038	&	13	&	7.7214E-06	\tabularnewline
4.00000035	&	4	&	8.8567E-08	&&14.00019356	&	14	&	1.3826E-05	\tabularnewline
5.00000063	&	5	&	1.2538E-07	&&15.00027199	&	15	&	1.8133E-05	\tabularnewline
6.00000124	&	6	&	2.0701E-07	&&16.00048015	&	16	&	3.0010E-05	\tabularnewline
7.00000252	&	7	&	3.5951E-07	&&17.00098217	&	17	&	5.7775E-05	\tabularnewline
8.00000365	&	8	&	4.5600E-07	&&18.00116124	&	18	&	6.4513E-05	\tabularnewline
9.00000790	&	9	&	8.7827E-07	&&19.00179071	&	19	&	9.4248E-05 \tabularnewline
\bottomrule
\end{tabular}
\end{table}

\subsubsection{Error Estimation}\label{subsubsec:ee}
In Subsection \ref{subsec:Discussion-of-Propositions}, we noted that the limit of $\norm{\tilde u_\lambda^{\b{\tilde N_2}}}_\H/\norm{\tilde u_\lambda^{\b{\tilde N_1}}}_\H$ either goes to one or infinity for $\tilde N_2>\tilde N_1$, depending on whether $\lambda$ is an eigenvalue (see Eqs. (\ref{eq:lim1}) and (\ref{eq:lim2}) for the specific limits). While the discussion in \ref{subsec:Discussion-of-Propositions} is for standard eigenvalues, the same limits apply for Steklov problems using Subsection \ref{subsec:Analysis-for-Steklov} instead of \ref{subsec:Convergence-Rate-Analysis}. We now attempt to use this fact to find intervals around local minima of $\norm{\tilde u_\lambda^{\b{\tilde N_2}}}_\H$ as a function of $\lambda$ that could contain the true eigenvalue. 

If we choose any cut-off $C>1$, the intervals where $\norm{\tilde u_\lambda^{\b{\tilde N_2}}}_\H/\norm{\tilde u_\lambda^{\b{\tilde N_1}}}_\H\le C$ will eventually narrow around the true eigenvalues. However, choosing a large value of $C$ may result in an interval too wide to provide useful information, and choosing $C$ too small may result in an empty set for smaller values of $\tilde{N}_2$; even at a true eigenvalue, this ratio is always greater than one (but approaches one in the limit). We may vary $C$ with $\tilde{N}_2$ as well, as long as $C\to 1$ no faster than the norm ratio for $\lambda=\lambda_k$. A possible choice is to let $C$ be the square of a local minimum of the norm ratio. Since this norm ratio decreases to 1 only for true eigenvalues, such an interval will shrink around the true eigenvalue but will never be empty. Using the square of the minimum produces an interval where the norm ratios are greater than one by a magnitude similar to the difference between the local minimum and one.

To demonstrate this, we compute the absolute error for the local minimum $\lambda^*$ closest to the true Steklov eigenvalues $\lambda=5$ and $10$, and the distance from the local minimum where the norm ratio becomes greater than the minimum ratio squared. These distances provide a rough error bound and are estimated by computing the ratio for $\lambda=\lambda^*\pm10^{-c}$ for $c=1.0,1.1,1.2,\ldots,7.0$. We use $\tilde N_1\approx0.9\tilde N_2$, where $N_2$ is again $\approx \b{\tilde N_\partial/4}^2$. 
Table \ref{tab:errest} shows the results of the error estimation test for $\lambda=5$ and $10$. Note that for $\lambda=10$ and $\tilde{N}_\partial=30$, all tested values of $\lambda$ resulted in a norm ratio below the cutoff, indicating that the eigenvalue estimate was not yet reliable. The estimated error bounds correctly bound the actual error and seem to approach zero rapidly, as desired.

\begin{table}[h]

\caption{Actual error and estimated error bound for the approximation of the $\lambda=5$ and $\lambda = 10$ Steklov eigenvalues of the disk}\label{tab:errest}
\begin{tabular}{@{}lllll@{}}
\toprule 
&\multicolumn{2}{@{}c@{}}{$\lambda=5$} & \multicolumn{2}{@{}c@{}}{$\lambda=10$} \\
$\tilde{N}_\partial$ & Absolute Error & Error Bound Estimate & Absolute Error & Error Bound Estimate\\
\midrule

30	&4.1881E-03&1.3589E-02&6.3149E-02&N/A\\
40 &1.1697E-04&3.1543E-04&1.1766E-03&1.5829E-02\\
50 &1.4624E-05&1.9753E-04&3.2470E-04&4.9119E-03\\
60 &1.0175E-06&6.2996E-05&2.2063E-05&2.5182E-04\\
70 &2.9521E-07&4.7247E-07&5.6377E-06&7.8433E-06\\

\bottomrule
\end{tabular}

\end{table}

\subsection{Exceptional Steklov--Helmholtz}

An advantage of this approach is that we can handle problems that
may be difficult with standard methods without any modifications.
Consider a Steklov--Helmholtz problem on the unit disk:
\begin{align*}
-\Delta u-\mu^{2}u & =0\text{ on }S,\\
\b{\hat{\vc n}\cdot\nabla u-\lambda u}\eval_{\partial S} & =0.
\end{align*}

A possible complication arises when $-\mu^{2}$
is an eigenvalue of the Laplacian with homogeneous Dirichlet boundary
conditions. In this case, there is a Steklov ``eigenvalue'' $\lambda$
at $-\infty$ associated with the Dirichlet eigenfunction(s); as $\lambda\to-\infty$, the $\b{\hat{\vc n}\cdot\nabla u-\lambda u}\eval_{\partial S}=0$
condition becomes $u\eval_{\partial S}=0$. When $-\mu^2$ is close to a Dirichlet eigenvalue (as is the case numerically due to finite precision), there can instead be a large, negative Steklov eigenvalue. This can introduce numerical instability for some methods, even for low-resolution computations.
Our method uses different matrices for each value of $\lambda$
(that can be formed from blocks of the same, more computationally
intensive matrices) and only tests solvability for a specific $\lambda$,
so it should not be affected by this issue.

To test this, we use $\mu=2.404825557695773$, which could cause the aforementioned
problem since $-\mu^2$ is a Dirichlet eigenvalue in this case. We estimate the first positive Steklov--Helmholtz
eigenvalue in Table \ref{tab:First-10-eigenvalues-1}, using the same
parameters as the previous subsection. We do not encounter any issues; there is nothing unique about the Dirichlet
eigenvalues for our method. Convergence in Table \ref{tab:First-10-eigenvalues-1}
is somewhat irregular, which is likely due to randomness in the point
cloud generation process.

\begin{table}[h]

\caption{\label{tab:First-10-eigenvalues-1}Estimated first positive Steklov--Helmholtz
eigenvalue on the unit disk, with $\mu=2.404825557695773$. The true
value is $\approx0.891592981473392$. The convergence rate is the
negative slope of the log-log plot of relative error plotted against
$\tilde{N}_{\partial}$}
\begin{tabular}{@{}llll@{}}
\toprule 
$\tilde{N}_{\partial}$ & $\lambda$ Estimate & Relative Error & Convergence\tabularnewline
\midrule

18	& 0.914785095121278	& 2.6012E-02
 & N/A\tabularnewline

30	& 0.891663966985991	& 7.9616E-05	& 11.333
\tabularnewline

42	& 0.891598053856351	& 5.6891E-06 & 7.842
\tabularnewline

54	& 0.891593063589785	& 9.2101E-08 & 16.407
\tabularnewline

66	& 0.891593015994537	& 3.8719E-08 & 4.318
\tabularnewline

78 & 0.891592987687612 & 6.9698E-09 & 10.265\tabularnewline
\bottomrule
\end{tabular}

\end{table}

\subsection{Schrödinger--Steklov}

We now consider the same problem as in the previous two subsections
but with a Schrödinger equation rather than $\Delta u=0$ or $-\Delta u-\mu^{2}u=0$.
That is, we consider
\begin{align*}
-\Delta u+qu & =0\text{ on }S,\\
\b{\hat{\vc n}\cdot\nabla u-\lambda u}\eval_{\partial S} & =0,
\end{align*}
where $q$ is a given function on $S$ (physically, it is the potential
energy). In \cite{quino18}, Quiñones computes asymptotic expressions
for the Steklov eigenvalues $\lambda$ in the case that $S$ is the unit
circle and $q$ is radial. These asymptotic expressions are compared
to numerically obtained eigenvalues computed by Quiñones using a finite
element method (FEM) based on code from Bogosel (Section 6 of \cite{bogos16}).
One of the radial functions considered by Quiñones is
\[
q\b r=\frac{\frac{1}{2}r+\frac{1}{5}\cos\b{5r}}{2r^{3}+1};
\]
the 10th non-zero %
eigenvalue with this function
was computed to be $\approx10.00807486$ (Table 2.2 of \cite{quino18}).
We repeat this finite element computation using Bogosel's code, modified
for the Schrödinger--Steklov problem, and compare convergence to our
meshfree approach. Table \ref{tab:stekschro} shows a convergence test with
all parameters identical to Subsection \ref{subsec:Steklov-Eigenvalues},
as well as convergence for the P2 finite element method as a comparison. The convergence rates given in Table \ref{tab:stekschro} are relative to $\tilde{N}_{\partial}$, which is
inversely proportional to point spacing and element size, for the
meshfree and FEM approaches, respectively.

\begin{table}[h]
\caption{\label{tab:stekschro}Schrödinger--Steklov eigenvalue $\lambda\approx10.00807486$
computed via minimization of $\protect\norm{u_{\lambda}^{\protect\b{\tilde{N}}}}_{\protect\H}^{2}$
(meshfree, left) compared to a FEM computation with P2 elements (right). The convergence rate is
 the negative slope of the log-log plot of relative error plotted against
 $\tilde{N}_{\partial}$
}

\begin{tabular}{@{}lllll@{}llll@{}}
\toprule 
\multicolumn{4}{@{}c@{}}{Meshfree Method} &{ } &\multicolumn{4}{@{}c@{}}{FEM P2 (Adapted from \cite{bogos16})}\\
\cmidrule{1-4} \cmidrule{6-9}
$\tilde{N}_{\partial}$ & $\tilde{N}$ & Relative Error & Convergence && $\tilde{N}_{\partial}$ & Vertices & Relative Error & Convergence\tabularnewline
\cmidrule{1-4} \cmidrule{6-9}

30 & 225 & 6.3316E-03 & N/A && 100 & 922 & 6.4788E-04 & N/A\tabularnewline

40 & 400 & 1.1775E-04 & 13.851 && 200 & 3592 & 7.5038E-05 & 3.110\tabularnewline

50 & 625 & 3.2598E-05 & 5.755 && 400 & 14002 & 1.2373E-05 & 2.600\tabularnewline

60 & 900 & 2.2258E-06 & 14.722 && 800 & 55300 & 2.6978E-06 & 2.197\tabularnewline

70 & 1225 & 5.7334E-07 & 8.799 && 1600 & 220477 & 6.4948E-07 & 2.054\tabularnewline
\bottomrule
\end{tabular}

\end{table}

Given that our method converges super-algebraically in theory,
the much faster observed convergence rate of the meshfree method is
to be expected; we are able to obtain accurate results with far fewer
points. Of course, this is at the cost of having a dense matrix, but
with the added benefit of being completely meshfree.

\subsection{Surface Steklov\label{subsec:Surface-Steklov}} 

A useful aspect of our approach is that the same method applies to
surface PDEs. This is not the case for various other high-order Steklov
eigenvalue approaches that may require a Green's function, which are
often not readily available for surfaces. We also do not require the
surface's metric or any quadrature scheme to be implemented;
only a point cloud and (unoriented) normal vectors on the point cloud
are needed.

We study a catenoid with a “wavy” edge, given by the parametrization:
\[
\sigma\b{s,t}=\b{\cosh\b t\cos\b s,\cosh\b t\sin\b s,t},
\]
for $s\in\left[0,2\pi\right)$ and 
$-1+0.1\sin\b{3s}\le t\le1+0.1\sin\b{3s}$. A point cloud for this surface is shown in Figure \ref{fig:catenoid}. The Steklov problem on this surface is
\begin{align*}
\Delta_{S}u & =0\text{ on }S,\\
\b{\hat{\vc{\nu}}\cdot\nabla_{S}u-\lambda u}\eval_{\partial S} & =0,
\end{align*}
where $\hat{\vc{\nu}}$ is the outward normal to $\partial S$ (perpendicular
to the tangent vector to $\partial S$ and the normal vector $\hat{\vc n}_{S}$
to $S$). This is discretized very similarly to the examples in Subsections
\ref{subsec:Laplace--Beltrami}-\ref{subsec:Steklov-Eigenvalues}:

\begin{align}
\text{minimize, over }u\in\H: & \norm u_{\H}\label{eq:eigsteklapbel},\\
\text{subject to } & \b{-\Delta u+\hat{\vc n}_{S}\cdot\b{D^{2}u}\hat{\vc n}_{S}}\b{\x_{j}}=0\text{, for each \ensuremath{j\in\cb{1,2,\ldots,\tilde{N}}}},\nonumber \\
 & \hat{\vc n}_{S}\cdot\nabla u\b{\vc x_{j}}=0\text{, for each }j\in\cb{1,2,\ldots,\tilde{N}},\nonumber \\
 & \hat{\vc n}_{S}\cdot\nabla u\b{\vc y_{j}}=0\text{, for each }j\in\cb{1,2,\ldots,\tilde{N}_{\partial}},\nonumber \\
 & \b{\hat{\vc{\nu}}\cdot\nabla u\b{\vc y_{j}}-\lambda u\b{\vc y_{j}}}=0\text{, for each }j\in\cb{1,2,\ldots,\tilde{N}_{\partial}},\nonumber \\
 & u\b{\vc a_{j}}=1\text{, for each }j\in\cb{1,2,\ldots,\tilde{N}_{a}}.\nonumber 
\end{align}

We seek to estimate the first non-zero eigenvalue. This is again done
via minimization of $\norm{u_{\lambda}^{\b{\tilde{N}}}}_{\H}^{2}$
as a function of $\lambda$. As it turns out, we see numerically that
the first eigenvalue is close to 0.46, so we initialize higher accuracy
tests starting at 0.46 before using Newton's method to approximately
find the minimum. We use $q=4$, $T=\tilde\ell=5$, and $\tilde{N}_a=1$. 
$\tilde{N}_{\partial}$
is the total number of points on the boundaries, so there are $\tilde{N}_{\partial}/2$
points on each of the two curves of $\partial S$. The eigenvalue
estimates for various $\tilde{N}_{\partial}$ (inversely proportional
to $h_{\text{max}}$) are in Table \ref{tab:stekschro-1}, along with
the change between subsequent tests. We observe convergence to 5 or 6 digits of accuracy with $\tilde{N}_{\partial}=126$. A plot of the eigenfunction for the $\tilde{N}_\partial=90$ test is shown in Figure \ref{fig:catenoid}.

\begin{table}[h]
\caption{\label{tab:stekschro-1}Steklov eigenvalue $\lambda\approx0.46506$
computed via minimization of $\protect\norm{u_{\lambda}^{\protect\b{\tilde{N}}}}_{\protect\H}^{2}$ and the difference between subsequent tests}
\begin{tabular}{@{}lll@{}}
\toprule 
$\tilde{N}_{\partial}$ & $\lambda$ Estimate & Relative Change\tabularnewline
\midrule
 
66 & 0.4651428 & N/A\tabularnewline

78 & 0.4650323 & 2.3756E-04\tabularnewline

90 & 0.4650468 & 3.1173E-05\tabularnewline

102 & 0.4650578 & 2.3630E-05\tabularnewline

114 & 0.4650583 & 9.7921E-07\tabularnewline

126 & 0.4650585 & 4.6817E-07\tabularnewline
\bottomrule
\end{tabular}

\end{table}
\begin{figure}[H]

\begin{centering}
\includegraphics[width=119mm]{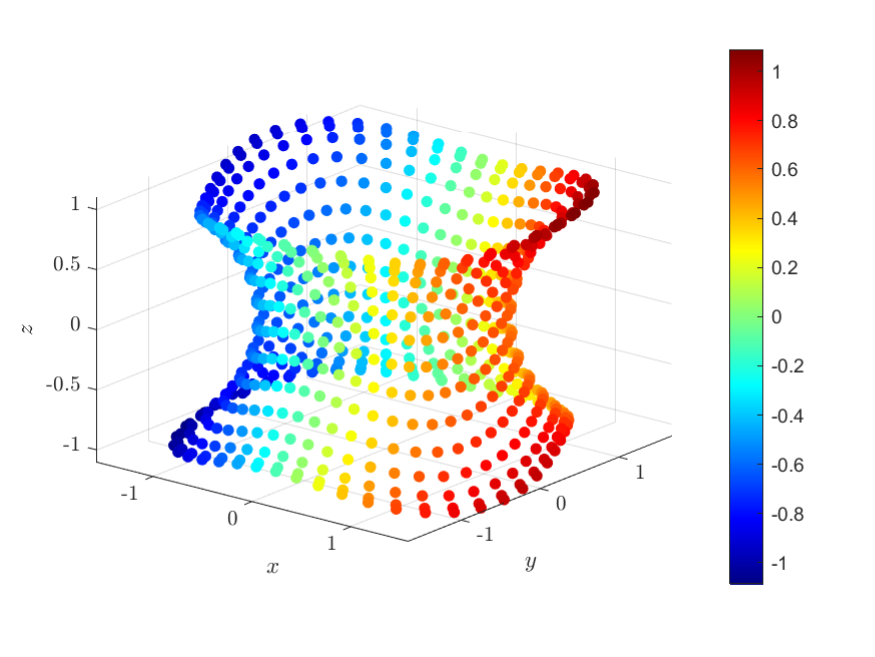}
\par\end{centering}
\caption{\label{fig:catenoid}Computed Steklov eigenfunction corresponding to $\lambda\approx0.465$ on the catenoid with a ``wavy'' boundary using $\tilde{N}_\partial=90$ boundary points}
\end{figure}

\subsection{Multiplicity}

We have seen how Proposition \ref{prop:Let--be} can be used to find
eigenvalues. With standard methods, eigenvalues of differential operators
are estimated as the eigenvalues of a matrix used to discretize the
differential operator. Such approaches can produce repeated eigenvalues,
which often correctly indicate the multiplicities of the eigenvalues
of the differential operator. In our approach, we simply see minima
of a function in $\R$ or $\C$. This leaves the problem of determining
the multiplicity of eigenvalues.

With probability one, if the multiplicity of an eigenvalue $\lambda$
is $n$, we expect to be able to interpolate $n$ random values at
$n$ random points with an eigenfunction with eigenvalue $\lambda$. We
can then consider the problem on a surface $S$:
\begin{align*}
\b{-\Delta_{S}u-\lambda u} & =0\text{ on }S,\\
u\b{\vc a_{j}} & =b_{j}\text{, for each }j\in\cb{1,2,\ldots,\tilde{N}_{a}}.
\end{align*}
Now, if the points $\cb{\vc a_{j}}$ and values $\cb{b_{j}}$ are
selected at random, we expect this to be solvable (with probability
one) only when $\lambda$ is an eigenvalue of $\Delta_{S}$ with multiplicity
at least $\tilde{N}_{a}$. There are a number of ways to use this
fact. On a surface, the discretized version of our problem is given
by (\ref{eq:eiglapbeltdisc}) from earlier. If $u_{\lambda,\tilde{N}_{a}}^{\b{\tilde{N}}}$
now represents the solution to (\ref{eq:eiglapbeltdisc}) with $\tilde{N}_{a}$
randomly selected points $\cb{\vc a_{j}}$ and values $\cb{b_{j}}$,
we expect $\norm{u_{\lambda,\tilde{N}_{a}}^{\b{\tilde{N}}}}_{\H}$
to only remain bounded as $\tilde{N}\to\infty$ when $\lambda$ is
an eigenvalue with multiplicity at least $\tilde{N}_{a}$.
We demonstrate this for the sphere in Fig. \ref{fig:multiplicity sphere},
where $\norm{u_{\lambda,\tilde{N}_{a}}^{\b{\tilde{N}}}}_{\H}$ is
plotted as a function of $\lambda$ for $\tilde{N}_{a}=1,3,$ and
5, using the same parameters as Subsection \ref{subsec:Laplace--Beltrami}.

\begin{figure}[H]
\begin{centering}
\includegraphics[width=119mm]{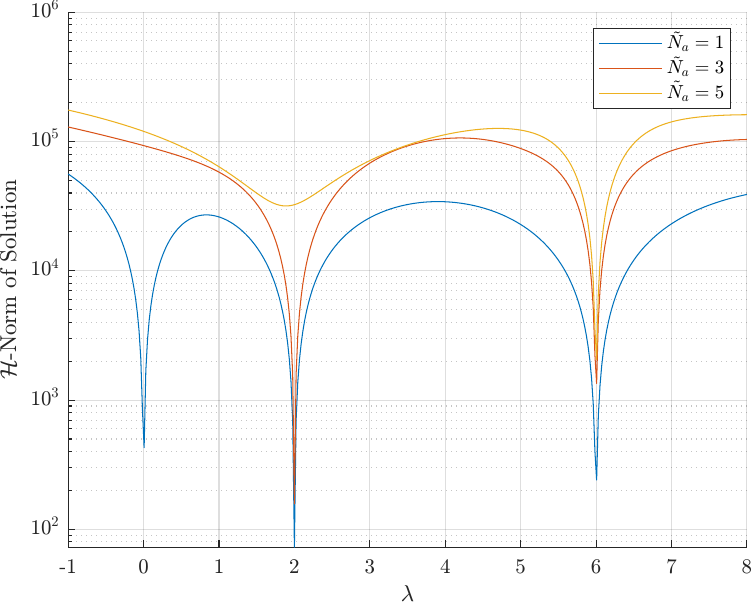}
\par\end{centering}
\caption{\label{fig:multiplicity sphere}$\protect\norm{u_{\lambda,\tilde{N}_{a}}^{\protect\b{\tilde{N}}}}_{\protect\H}$
plotted against $\lambda$ to test for eigenvalue multiplicities,
using $\tilde{N}_{a}=1,3,5$. $\tilde{N}=400,N_{b}=\protect\b{2\cdot15+1}^{3},\Omega=\protect\sb{-2,2}^{3},q=4$, and $T=4$,
with the same $\mathcal{H}$ as Subsection \ref{subsec:Laplace--Beltrami}}
\end{figure}

Visually, we see that there are sharp minima when $\lambda$ is an
eigenvalue \textit{and} it has multiplicity of at least $\tilde{N}_{a}$.
In this example, $\lambda=0$ has multiplicity 1, $\lambda=2$ has
multiplicity $3$, and $\lambda=6$ has multiplicity $5$. However, we
see from the plot that for $\tilde{N}_{a}=5$, there may be another
minimum of $\norm{u_{\lambda,\tilde{N}_{a}}^{\b{\tilde{N}}}}_{\H}$
somewhere between $\lambda=1$ and $\lambda=2$. This minimum does
not look as sharp as the others, however, which motivates us to consider
a more robust test of multiplicity that can correctly distinguish
these two types of local minima. To do this, we use the limits from Equations (\ref{eq:lim1}) and
(\ref{eq:lim2}). These limits are quite useful since they tell us
that the ratio $\norm{u_{\lambda,\tilde{N}_{a}}^{\b{\tilde{N}_{2}}}}_{\H}/\norm{u_{\lambda,\tilde{N}_{a}}^{\b{\tilde{N}_{1}}}}_{\H}$
must go to either 1 or $\infty$, depending on whether the problem
is solvable or not. Once we already have eigenvalue estimates using
$\tilde{N}_{a}=1$, we can check their multiplicity using the norm
ratio. 

As a demonstration, we look at the $\lambda=56$ eigenvalue for the
Laplace--Beltrami operator on the unit sphere, which has multiplicity
15. We give the norm ratio $\norm{u_{\lambda,\tilde{N}_{a}}^{\b{\tilde{N}_{2}}}}_{\H}/\norm{u_{\lambda,\tilde{N}_{a}}^{\b{\tilde{N}_{1}}}}_{\H}$
for $\lambda=56$ and multiplicities 14-17 in Table \ref{tab:Norm-ratio-for}.
We use $\tilde{N}_{2}\approx\frac{10}{9}\tilde{N}_{1}$ and
the same $\H$ as in the previous test. The expected behaviour is observed, and by the last test, the true multiplicity
is fairly clearly indicated. That is, the norm ratio for $\tilde{N}_{a}=15$
is approaching 1, while the norm ratio for $\tilde{N}_{a}=16$ appears
to be diverging; the norm is still nearly doubling with each (fairly
small) refinement. This indicates to us that it is likely possible
to select a cutoff value for the ratio slightly higher than one, then
consider a ratio less than that value to indicate an eigenvalue with
at least multiplicity $\tilde{N}_{a}$.

\begin{table}[h]
\caption{\label{tab:Norm-ratio-for}Norm ratio $\protect\norm{u_{\lambda,\tilde{N}_{a}}^{\protect\b{\tilde{N}_{2}}}}_{\protect\H}/\protect\norm{u_{\lambda,\tilde{N}_{a}}^{\protect\b{\tilde{N}_{1}}}}_{\protect\H}$for
$\lambda=56$ (multiplicity 15) and various $\tilde{N}_{1}$ values, with $\tilde{N}_{2}\approx\frac{10}{9}\tilde{N}$}

\begin{tabular}{@{}llllll@{}}
\toprule 

\diagbox{$\tilde{N}_{1}$}{$\tilde{N}_a$} && 14 & 15 & 16 & 17\\

\midrule 
400 && 1.0112 & 1.0496 & 1.0897 & 1.4007\tabularnewline

500 && 1.0026 & 1.0628 & 1.2488 & 1.3880\tabularnewline

600 && 1.0005 & 1.0350 & 1.7106 & 1.6857\tabularnewline

700 && 1.0001 & 1.0073 & 1.6638 & 1.7111\tabularnewline

800 && 1.0000 & 1.0026 & 1.9056 & 2.1866\tabularnewline
\bottomrule
\end{tabular}

\end{table}

Using $\tilde{N}_{2}\approx\frac{10}{9}\tilde{N}_{1}$, we also test
the multiplicity of the $\lambda=156$ eigenvalue in Table \ref{tab:Norm-ratio-for-1}. $\tilde{N}_{1}=900$ or $1000$ offers sufficient resolution for us to observe that the ratio is approaching 1 for $\tilde{N}_{a}=25$,
but increasing for $\tilde{N}_{a}=26$. Note that this test covers
the 625th eigenvalue, including multiplicity.

\begin{table}[h]
\caption{\label{tab:Norm-ratio-for-1}Norm ratio $\protect\norm{u_{\lambda,\tilde{N}_{a}}^{\protect\b{\tilde{N}_{2}}}}_{\protect\H}/\protect\norm{u_{\lambda,\tilde{N}_{a}}^{\protect\b{\tilde{N}_{1}}}}_{\protect\H}$for
$\lambda=156$ (multiplicity 25) and various $\tilde{N}_{1}$ values, with $\tilde{N}_{2}\approx\frac{10}{9}\tilde{N}$}
\begin{tabular}{@{}lllll@{}}
\toprule 
 \diagbox{$\tilde{N}_{1}$}{$\tilde{N}_a$} 
& 24 & 25 & 26 & 27 \\
\midrule

400 & 1.1397 & 1.1473 & 1.1501 & 1.1509\tabularnewline

500 & 1.2410 & 1.2385 & 1.2933 & 1.3034\tabularnewline
 
600 & 1.2108 & 1.2349 & 1.3843 & 1.3566\tabularnewline
 
700 & 1.1598 & 1.2182 & 1.5749 & 1.5796\tabularnewline

800 & 1.0979 & 1.2434 & 1.8705 & 1.8795\tabularnewline

900 & 1.0277 & 1.0811 & 2.0428 & 2.1348\tabularnewline

1000 & 1.0074 & 1.0232 & 2.2747 & 2.2875\tabularnewline

\bottomrule
\end{tabular}

\end{table}

\section{Conclusions} \label{sec:conclude}

We presented a very general result (Proposition \ref{prop:Let--be})
detailing how the boundedness of Hermite--Birkhoff interpolants in certain
Hilbert spaces is necessary and sufficient for solutions to linear
PDEs to exist in a very general context, and we explained how this could
be used to determine the solvability of linear PDEs. In Propositions
\ref{prop:If-all-assumptions} and \ref{prop:If-all-assumptions-1}, we proved inequalities that show the high-order convergence
of our approach for estimating eigenvalues and eigenfunctions. Then,
we tested our method numerically for a variety of problems and observed the
rapid convergence of eigenvalue estimates. Notably, we are able to
handle surface PDEs, irregular domains with holes, problems with varying coefficients, and Steklov problems all with
the same approach.%

Our method has certain advantages; we see its universality for linear
PDEs as its primary advantage, as well as its meshfree nature. Propositions
\ref{prop:Let--be}, \ref{prop:If-all-assumptions}, and \ref{prop:If-all-assumptions-1}
show analytically that the method produces correct eigenvalues with
no spurious modes, and show that the method converges at a high-order
rate for suitable problems. We also observe that our method is numerically
reliable for producing correct eigenvalues at high enough point densities,
and that estimates converge extremely quickly. This differs from other
high-order, meshfree methods that have been used for solving PDEs,
but lack analytical convergence results for eigenvalue estimates and
do not reliably produce the correct eigenvalues without spurious modes
in practice. Meshfree methods are highly desirable for surface PDEs due to the difficulty of mesh creation. It is generally much easier
to produce a point cloud than a mesh for irregularly shaped flat domains and surfaces. In practice, surfaces may also be defined by a point cloud originating from a scan of an object, which requires a large amount of pre-processing to mesh.

Extensions of this work may focus on scaling up the method to solve larger problems.
There are Hilbert spaces with useful basis functions properties (compact
support, separable, etc.) that can be used to greatly decrease computational
costs. In other work, we are investigating Hilbert spaces that produce
$\psi_{j}$ functions that vary by location, which may help substantially
in capturing fine details where necessary without substantially increasing
the point density for the whole domain. This may help with a key weakness
of standard global RBF methods, where narrower basis functions cannot
be used without high point densities for the entire domain. We are also 
exploring using Proposition \ref{prop:Let--be} for other problems regarding 
PDE solvability, such as inverse problems, and using known PDE solvability conditions that depend on integrals to develop accurate meshfree integration techniques on surfaces.

\begin{appendices}
\section{Point Cloud Generation Algorithms}
\subsection{Boundary Algorithm for Curves and Surfaces}\label{alg:boundary}
Let $\tilde{N}_{\partial}$ be the desired number of points on the boundary $\partial U$ of an open set $U$.
\begin{enumerate}
\item Start with an empty point cloud.
\item Create $\tilde{N}_{\text{test},\partial}$ ($>\tilde{N}_{\partial}$ for Version A) points
on the boundary: $\cb{\vc z_{j}}_{j=1}^{\tilde{N}_{\text{test},\partial}}$.
If the boundary is defined by a level set, this can be achieved by placing points
in a larger set containing $\partial U$, then using Newton's method
or another root-finding algorithm to move the points onto $\partial U$. That is, if $\partial U:=\cb{\x\in\R^m:\vp\b\x=0}$, we can initialize a root-finding algorithm at a point $\tilde{\vc z}_j$ near $\partial U$ to find some $\vc z_j\in\partial U$ such that $\vp \b{\vc z_j}=0$.
\item If the boundary point cloud has $k$ points, $\cb{\vc y_{j}}_{j=1}^{k}$,
choose $\vc y_{k+1}$ to be the point in $\cb{\vc z_{j}}_{j=1}^{\tilde{N}_{\text{test},\partial}}$
farthest away from $\cb{\vc y_{j}}_{j=1}^{k}$ (simply using the Euclidean
norm in the embedding space: $\norm{\cdot}_{2}$).
\item Repeat the previous step (Version A) or repeat steps 2 and 3 (Version B) until there are $\tilde{N}_{\partial}$ points
in the point cloud: $\cb{\vc y_{j}}_{j=1}^{\tilde{N}_{\partial}}$.
\end{enumerate}
\subsection{Interior Algorithm}\label{alg:interior}
Let $U\subset\R^m$ be open and defined by a level set $U:=\cb{\x\in\R^{m}:\vp\b{\vc x}<0}$, and let
the minimum value of $\vp$ on $\overline{U}$ be $a$.
\begin{enumerate}
\item Start with an empty point cloud.
\item Create $\tilde{N}_{\text{test}}$ points in $U$: $\cb{\vc z_{j}}_{j=1}^{\tilde{N}_{\text{test}}}$.
\item If the point cloud currently has $k$ points, $\cb{\vc x_{j}}_{j=1}^{k}$,
choose $\vc x_{j}$ to be the point in $\cb{\vc z_{j}}_{j=1}^{\tilde{N}_{\text{test}}}$
with the largest penalty, where the penalty is defined by %
\[
P_{k}\b{\vc z}=\b{w\b{1-\frac{\vp\b{\vc z}}{a}}+1}\min\cb{\norm{\x-\vc z}_{2}^{2}:\vc x\in\cb{\vc x_{j}}_{j=1}^{k}\cup\cb{\vc y_{j}}_{j=1}^{\tilde{N}_{\partial}}},
\]
where $w$ is a parameter that controls preference for placing points
near the boundary.
\item Repeat steps 2 and 3 until the point cloud has $\tilde{N}$ points.
\end{enumerate}

\end{appendices}

\subsection*{Acknowledgements}
We acknowledge the support of the Natural Sciences and Engineering Research Council of Canada (NSERC), [funding reference number RGPIN 2022-03302].

\printbibliography

\end{document}